\documentclass[12pt]{amsart}

\usepackage{amsmath,amsthm,amssymb,amsfonts,bbm,latexsym,xcolor,mathrsfs,enumerate,array}

\usepackage[hidelinks]{hyperref}

\allowdisplaybreaks

\numberwithin{equation}{section}

\title[Moments of zeta and correlations of divisor-sums]{Moments of zeta and correlations of divisor-sums: stratification and Vandermonde integrals
}
\author{Siegfred Baluyot}
\address{American Institute of Mathematics \\
600 East Brokaw Road San Jose, CA 95112}
\email{\href{mailto:sbaluyot@aimath.org}{sbaluyot@aimath.org}}
\author{Brian Conrey}
\address{American Institute of Mathematics \\
600 East Brokaw Road San Jose, CA 95112}
\email{\href{mailto:conrey@aimath.org}{conrey@aimath.org}}


\newcommand{\ordp}{\text{ord}_p}
\DeclareMathOperator{\re}{\text{\upshape{Re}}}

\newtheorem{theorem}{Theorem}[section]
\newtheorem{conjecture}[theorem]{Conjecture}
\newtheorem{lemma}[theorem]{Lemma}

\begin{document}

\begin{abstract}
We refine a recent heuristic developed by Keating and the second author. Our improvement leads to a new integral expression for the conjectured asymptotic formula for shifted moments of the Riemann zeta-function. This expression is analogous to a formula, recently discovered by Brad Rodgers and Kannan Soundararajan, for moments of characteristic polynomials of random matrices from the unitary group.
\end{abstract}

\maketitle

\section{Introduction}

One of the most important problems in analytic number theory is to find an asymptotic formula for the $2k$th moment
$$
M_k(T) := \int_0^T |\zeta(\tfrac{1}{2}+it)|^{2k}\,dt
$$
of the Riemann zeta-function $\zeta(s)$, where $k$ is an arbitrary positive integer. A folklore conjecture suggests that, for some unspecified constant $c_k$, we have $M_k(T)\sim c_k T(\log T)^{k^2}$ as $T\rightarrow \infty$. To date, this is known only for $k=1,2$, with $c_1=1$ and $c_2=1/(2\pi^2)$~\cite{T}. The problem is so intractable that, up until recently, there had been no viable guess for the exact value of $c_k$. The approach of using correlations of divisor sums leads to conjectures for $c_3$ and $c_4$~\cite{conreyghosh,conreygonek}, and the process has recently been examined in close detail and made more precise by Ng~\cite{ng}, Hamieh and Ng~\cite{hamiehng}, and Ng, Shen, and Wong~\cite{ngshenwong}. This approach seems to fail, however, to give a reasonable guess for $c_k$ when $k\geq 5$~\cite{conreygonek}.

A breakthrough was made when Keating and Snaith~\cite{keatingsnaith} used random matrix theory to predict the exact value of $c_k$ for all complex $k$ with $\re(k)\geq -1/2$. Remarkably, their predicted values of $c_3$ and $c_4$ agree with the conjectures in \cite{conreyghosh} and \cite{conreygonek}. The conjectures for $c_k$ for positive integers $k$ were then refined by Farmer, Keating, Rubinstein, Snaith, and the second author~\cite{cfkrs} through a heuristic method called the \textit{recipe}, which also applies to a general family of $L$-functions. At about the same time, Diaconu, Goldfeld, and Hoffstein~\cite{diaconugoldfeldhoffstein} predicted the same values of $c_k$ via a different approach using multiple Dirichlet series. Despite the differences between these approaches, all the conjectures agree.

The recipe arrives at the conjecture by using the approximate functional equation and then predicting that certain ``off-diagonal'' terms cancel in the evaluation of the moment. However, it does not indicate how these terms cancel or combine. In a recent series of papers~\cite{CK1,CK2,CK3,CK4,CK5}, Keating and the second author address this problem by revisiting the now conventional approach of using Dirichlet polynomial approximations to $\zeta^k(s)$ and examining correlations of divisor sums. This approach, which relies on the delta method of Duke, Friedlander, and Iwaniec~\cite{dfi}, was previously employed by Gonek and the second author~\cite{conreygonek} to conjecture the values of $c_3$ and $c_4$ from a number theoretic perspective. The principal result of \cite{CK5} is a new heuristic method that indicates how divisor sums may be combined to recover the prediction of the recipe. This new heuristic is inspired by ideas of Bogomolny and Keating~\cite{bogomolnykeating1,bogomolnykeating2}, and is reminiscent of the Hardy-Littlewood circle method.

In this paper, we refine the approach of \cite{CK5} by using an integral form of the asymptotic formula for correlations of divisor sums that is predicted by the delta method. Through this and a generalization of the local calculations in \cite{CK5}, we predict that combining the divisor sums in the same way as in \cite{CK5} leads to a certain ``Vandermonde integral'' expression (Conjecture~\ref{conj: vandermonde} below). Evaluating this Vandermonde integral then immediately gives the sum of the $\ell$-swap terms from the recipe prediction (Conjecture~\ref{conj: ell-swaps} below), and thus removes the need to examine all possible decompositions of $A$ and $B$ as in Section~12 of \cite{CK5}.

We are interested in the shifted moments
\begin{equation*}
\frac{1}{T}\int_T^{2T} \prod_{\alpha\in A }\zeta(\tfrac{1}{2}+\alpha +it) \prod_{\beta\in B }\zeta(\tfrac{1}{2}+\beta -it)\,dt,
\end{equation*}
where $T$ is a parameter tending to $\infty$ and $A$ and $B$ are finite multisets of complex numbers that have small moduli (say $\ll 1/\log T$). We study these moments by examining their Dirichlet polynomial approximations
\begin{equation}\label{eqn: dirichletpolynomialapprox}
\mathcal{M}_{A,B}(T,X) := \frac{1}{T}\int_0^{\infty} \psi\left( \frac{t}{T}\right) \sum_{m=1}^{\infty} \frac{\tau_A (m) }{m^{\frac{1}{2}+it}} \Upsilon\left( \frac{m}{X} \right)\sum_{n=1}^{\infty} \frac{\tau_B (n) }{n^{\frac{1}{2}-it}} \Upsilon\left( \frac{n}{X} \right) \,dt ,
\end{equation}
where $X$ is another parameter tending to $\infty$, $\psi$ is a smooth, nonnegative function that is supported on $[1,2]$, say, and $\Upsilon$ is a smooth, nonnegative function that is supported on $[0,1]$, say, and satisfies $\Upsilon(0)>0$. Here, the coefficients $\tau_A$ are defined for finite multisets $A$ by
\begin{equation}\label{taudef}
\sum_{m=1}^{\infty}\frac{\tau_A(m)}{m^s} =\prod_{\alpha\in A} \zeta(s+\alpha)
\end{equation}
for all $s$ such that the left-hand side converges, where the product on the right-hand side is over all $\alpha\in A$, counted with multiplicity. The recipe of \cite{cfkrs} leads to the prediction (see Section~\ref{sec: recipe})
\begin{align}
\mathcal{M}_{A,B} (T,X) \sim \sum_{\substack{U\subseteq A, V\subseteq B \\  |U|=|V| }} &  \frac{1}{(2\pi i )^2} \int_{(\varepsilon)} \int_{(\varepsilon)} \widetilde{\Upsilon}(\xi) \widetilde{\Upsilon}(\eta) X^{\xi + \eta}  \frac{1}{T}\int_0^{\infty} \psi\left(\frac{t}{T}\right)  \notag \\
& \times  \prod_{\alpha\in U}\chi(\tfrac{1}{2}+\xi+\alpha+it)  \prod_{\beta\in V}\chi (\tfrac{1}{2}+\eta+\beta-it) \notag \\
& \hspace{.25in} \times \sum_{n=1}^{\infty} \frac{\tau_{(A \smallsetminus U)_{\xi}\cup (V_{\eta})^- }(n) \tau_{ (B\smallsetminus V)_{\eta} \cup  (U_{\xi})^-}(n)}{ n}    \,d\xi\,d\eta, \label{recipe}
\end{align}
where $\varepsilon>0$ is an arbitrarily small constant, $\widetilde{\Upsilon}$ denotes the Mellin transform of $\Upsilon$, $\chi$ is the factor from the functional equation $\zeta(s)=\chi(s)\zeta(1-s)$, and the $n$-sum should be interpreted as its analytic continuation. Here, we use the notations $A_s:=\{\alpha+s : \alpha\in A \}$ and $A^{-}:=\{-\alpha : \alpha\in A\}$ for a multiset $A$ and a complex number $s$. Notice that if $|U|=|V|=\ell$, then we are exchanging $\ell$ elements from $A_{\xi}$ with $\ell$ elements from $(B_{\eta})^{-}$ in forming the sets $(A\smallsetminus U)_{\xi}\cup(V_{\eta})^-$ and $(B\smallsetminus V)_{\eta}\cup(U_{\xi})^-$. Thus, we may refer to the terms with $|U|=|V|=\ell$ in \eqref{recipe} as the ``$\ell$-swap terms.''

Our goal is to understand how and what properties of divisor sums might lead to \eqref{recipe}. We let $\ell$ be a fixed integer with $1\leq \ell\leq \min\{|A|,|B|\}$, and partition $A$ into $\ell$ nonempty sets $A=A_1\cup\cdots \cup A_{\ell}$ and similarly write $B=B_1\cup\cdots \cup B_{\ell}$. We may then express $\tau_A$ and $\tau_B$ as the Dirichlet convolutions $\tau_A=\tau_{A_1}*\cdots *\tau_{A_{\ell}}$ and $\tau_B=\tau_{B_1}*\cdots *\tau_{B_{\ell}}$, and deduce from \eqref{eqn: dirichletpolynomialapprox} that
\begin{align}
\mathcal{M}_{A,B}(T,X) = \frac{1}{T}\int_0^{\infty} \psi\left( \frac{t}{T}\right) \sum_{\substack{ 1 \leq m_1,\dots,m_{\ell} <\infty \\ 1\leq n_1,\dots,n_{\ell} <\infty}  }  \frac{\tau_{A_1} (m_1)\cdots \tau_{A_{\ell}} (m_{\ell}) \tau_{B_1} (n_1)\cdots \tau_{B_{\ell}} (n_{\ell}) }{(m_1 \cdots m_{\ell} )^{\frac{1}{2}+it} (n_1 \cdots n_{\ell})^{\frac{1}{2}-it}} \notag\\
\times \Upsilon \left(\frac{m_1\cdots m_{\ell}}{X}\right) \Upsilon \left(\frac{n_1\cdots n_{\ell}}{X}\right) \,dt. \label{eqn: dirichletpolynomialapprox2}
\end{align}
The basic idea behind the approach in \cite{CK5} is to sum over the $m_j,n_j$ with $m_j/n_j$ close to the ``rational direction'' $M_j/N_j$ and consider all possible directions subject to the natural conditions $(M_j,N_j)=1$ and $M_1\cdots M_{\ell}=N_1\cdots N_{\ell}$. A key step in the approach is to use the delta method of \cite{dfi} to evaluate each $m_j,n_j$-sum, and then combine the results. Our starting point is the sum
\begin{align}
\mathcal{S}_{\ell} : = \frac{1}{(\ell !)^2}
& \sum_{\substack{M_1\cdots M_{\ell} = N_1\cdots N_{\ell} \\ (M_j,N_j)=1 \ \forall j}} \sum_{ \substack{h_1,\dots,h_{\ell} \in \mathbb{Z} \\ h_1\cdots h_{\ell}\neq 0} } \frac{1}{T}\int_0^{\infty} \psi\left( \frac{t}{T}\right) \notag\\
& \times \sum_{\substack{ 1 \leq m_1,\dots,m_{\ell} <\infty \\ 1\leq n_1,\dots,n_{\ell} <\infty  \\ m_jN_j-n_jM_j = h_j \ \forall j }  }  \frac{\tau_{A_1} (m_1)\cdots \tau_{A_{\ell}} (m_{\ell}) \tau_{B_1} (n_1)\cdots \tau_{B_{\ell}} (n_{\ell}) }{(m_1 \cdots m_{\ell} )^{\frac{1}{2}+it} (n_1 \cdots n_{\ell})^{\frac{1}{2}-it}} \notag\\
& \hspace{.5in} \times \Upsilon \left(\frac{m_1\cdots m_{\ell}}{X}\right) \Upsilon \left(\frac{n_1\cdots n_{\ell}}{X}\right) \,dt. \label{Sldef}
\end{align}
We also define
\begin{align*}
\mathcal{S}_0: = \frac{1}{(2\pi i )^2} \int_{(\varepsilon)} \int_{(\varepsilon)} \widetilde{\Upsilon}(\xi) \widetilde{\Upsilon}(\eta) X^{\xi + \eta}  \frac{1}{T}\int_0^{\infty} \psi\left(\frac{t}{T}\right) \sum_{n=1}^{\infty} \frac{\tau_A(n) \tau_B(n)}{ n^{1+\xi+\eta}}    \,d\xi\,d\eta.
\end{align*}
We expect the following.
\begin{conjecture}\label{conj: stratification} If $\alpha,\beta \ll 1/\log T$ for each $\alpha\in A$ and $\beta\in B$, then as $T\rightarrow \infty$ we have
\begin{equation*}
\mathcal{M}_{A,B}(T,X) \sim \sum_{\ell=0}^{\min\{|A|,|B|\}} \mathcal{S}_{\ell}.
\end{equation*}
\end{conjecture}

One way to view this paper is that it gives evidence for this conjecture, as we predict (in Conjecture~\ref{conj: ell-swaps} below) that $\mathcal{S}_{\ell}$ is essentially the sum of the $\ell$-swap terms from \eqref{recipe}. Towards this, we use an integral form of the asymptotic formula for correlations of divisor sums that is predicted by the delta method (see \eqref{additivedivisorhypothesis} below) and rigorous evaluations of the local factors of an Euler product (Theorem~\ref{euler} below) to predict the following ``Vandermonde integral'' expression for $\mathcal{S}_{\ell}$.

\begin{conjecture}\label{conj: vandermonde} Let $\ell$ be an integer with $1\leq \ell \leq \min\{|A|,|B|\}$. If $\alpha,\beta \ll 1/\log T$ for each $\alpha\in A$ and $\beta\in B$, then as $T\rightarrow \infty$ we have
\begin{align}
\mathcal{S}_{\ell} \sim & \frac{1}{(\ell!)^2 (2\pi i )^2} \int_{(2\varepsilon)} \int_{(2\varepsilon)} \widetilde{\Upsilon}(\xi) \widetilde{\Upsilon}(\eta) \frac{X^{\xi + \eta}}{T} \int_0^{\infty} \psi\left( \frac{t}{T}\right) \frac{1}{(2\pi i)^{2\ell}} \oint_{|z_1|=\varepsilon} \cdots \oint_{|z_{\ell}|=\varepsilon}   \notag\\
& \hspace{.2in} \times \oint_{|w_1|=\varepsilon} \cdots \oint_{|w_{\ell}|=\varepsilon} \prod_{j=1}^{\ell} \Big\{  \chi (\tfrac{1}{2} +\xi-z_j+it) \chi (\tfrac{1}{2} +\eta-w_j-it)\Big\}\notag\\
& \hspace{.4in} \times \prod_{\substack{\alpha\in A \\ \beta\in B}} \zeta(1+\alpha+\beta+\xi+\eta)\prod_{\substack{1\leq j\leq \ell \\ \alpha\in A}}\zeta(1+\alpha+z_j) \prod_{\substack{1\leq j\leq \ell \\ \beta\in B}}\zeta(1+\beta+w_j) \notag\\
& \hspace{.6in} \times \prod_{\substack{1\leq j\leq \ell \\ \alpha\in A}} (1/\zeta)(1+ \alpha+\xi+\eta-w_j)\prod_{\substack{1\leq j\leq \ell \\ \beta\in B}} (1/\zeta) (1 +\beta+\xi+\eta-z_j)  \notag\\
& \hspace{.8in} \times \prod_{\substack{1\leq i,j\leq \ell \\ i\neq j}} (1/\zeta) (1-z_i+z_j) \prod_{\substack{1\leq i,j\leq \ell \\ i\neq j}} (1/\zeta) (1-w_i+w_j)  \notag\\
& \hspace{1in} \times \prod_{\substack{1\leq i,j\leq \ell}}\zeta(1 +z_i+w_j-\xi-\eta)\zeta(1 -z_i-w_j+\xi+\eta) \notag\\
& \hspace{1.2in} \times \mathcal{A}(A,B,Z,W,\xi+\eta) \ dw_{\ell}\cdots dw_1 \,dz_{\ell}\cdots dz_1\,d\xi\,d\eta, \label{eqn: vandermonde}
\end{align}
where $Z:=\{z_1,\dots,z_{\ell}\}$, $W:=\{w_1,\dots,w_{\ell}\}$, and $\mathcal{A}(A,B,Z,W,\xi+\eta)$ is an Euler product that converges absolutely whenever $\,\re(\xi)=\re(\eta)=2\varepsilon$ and $\,|\re(\gamma)|\leq \varepsilon$ for all $\gamma\in A\cup B\cup Z\cup W$. Explicitly, $\mathcal{A}$ is defined by \eqref{mathcalAdef} below.
\end{conjecture}

After discovering this Vandermonde integral expression through a rough ``back-of-the-envelope'' calculation, the second author informed Brad Rodgers of it in the summer of 2019. Rodgers responded that he and Kannan Soundararajan had previously found an analogous expression for moments of characteristic polynomials of random matrices. They had proved that if $U(N)$ is the group of $N\times N$ unitary matrices, then integrating with respect to the Haar measure gives~\cite{rodgers}
\begin{align*}
\int_{U(N)} \prod_{\alpha\in A} \det\big( 1-e^{-\alpha}g\big) \prod_{\beta\in B}\det \big( 1-e^{-\beta}g^{-1}\big) \,dg = \sum_{\ell=0}^{\min\{|A|,|B|\}} J_{\ell}^{A,B},
\end{align*}
where, for positive integers $\ell$, $J_{\ell}^{A,B}$ is defined by
\begin{align*}
J_{\ell}^{A,B} :=
& \frac{(-1)^{\ell}}{(\ell!)^2(2\pi i)^{2\ell+1}} \mathop{\oint}_{|\xi|=1} \mathop{\oint \cdots\oint}_{\substack{|z_1|=\cdots=|z_{\ell}|=\varepsilon\\ |w_1|=\cdots=|w_{\ell}|= \varepsilon}} \frac{e^{-(N+\ell)\sum_{i=1}^{\ell} (z_i+w_i)}}{1-e^{-\xi} } \\
&  \hspace{.25in} \times \frac{\mathcal{Z}(A,B)\mathcal{Z}(A,Z^{-})\mathcal{Z}(B,W^{-})     }{   \mathcal{Z}(A,W_{\xi})     \mathcal{Z}(B,Z_{\xi})  } \tilde{\Delta}_{\ell}(Z)\tilde{\Delta}_{\ell}(W) \mathcal{Z}(W,Z_{\xi})^2\\
&  \hspace{.5in} \times  dw_{\ell}\cdots dw_1 \,dz_{\ell}\cdots dz_1\,d\xi,
\end{align*}
and $J_0^{A,B}$ is defined by $J_0^{A,B}:=\mathcal{Z}(A,B)$, with
$$
\mathcal{Z}(C,D):= \prod_{\substack{\gamma\in C\\ \delta\in D}}\frac{1}{1-e^{-\gamma-\delta}}
$$
and
$$
\tilde{\Delta}_{\ell}(C):=\prod_{\substack{\gamma,\hat{\gamma}\in C \\ \gamma\neq \hat{\gamma}}} (1-e^{\gamma-\hat{\gamma}})
$$
for finite multisets $C,D$ of complex numbers. It is quite remarkable that these two analogous Vandermonde integral expressions were discovered independently at about the same time through different approaches. In hindsight, it is straightforward to show that the sum of the $\ell$-swap terms from \eqref{recipe} equals the right-hand side of \eqref{eqn: vandermonde} once one has already seen the right-hand side of \eqref{eqn: vandermonde} and knows what to aim for. We do this calculation at the end of Section~\ref{sec: vandermondeintegral} and arrive at the following prediction.
\begin{conjecture}\label{conj: ell-swaps} Let $\ell$ be an integer with $1\leq \ell \leq \min\{|A|,|B|\}$. If $\alpha,\beta \ll 1/\log T$ for each $\alpha\in A$ and $\beta\in B$, then as $T\rightarrow \infty$ we have
\begin{align}
\mathcal{S}_{\ell} \sim
& \sum_{\substack{ U\subseteq A, V\subseteq B\\ |U|=|V|={\ell} }}\frac{1}{(2\pi i )^2} \int_{(\varepsilon)} \int_{(\varepsilon)} \widetilde{\Upsilon}(\xi) \widetilde{\Upsilon}(\eta) \frac{X^{\xi + \eta}}{T} \int_0^{\infty} \psi\left( \frac{t}{T}\right)  \notag\\
& \hspace{.25in} \times \prod_{\alpha\in U}   \chi (\tfrac{1}{2} +\xi+\alpha+it) \prod_{\beta\in V}\chi (\tfrac{1}{2} +\eta+\beta-it)  \notag\\
& \hspace{.5in}  \times \sum_{n=1}^{\infty} \frac{\tau_{(A \smallsetminus U)_{\xi}\cup (V_{\eta})^- }(n) \tau_{ (B\smallsetminus V)_{\eta} \cup  (U_{\xi})^-}(n)}{ n} \,d\xi\,d\eta. \notag
\end{align} 
\end{conjecture}

Note that the right-hand side is exactly the sum of the $\ell$-swap terms from \eqref{recipe}. Thus, the recipe prediction \eqref{recipe} and Conjecture~\ref{conj: ell-swaps} lead us to believe Conjecture~\ref{conj: stratification}.

In an AIM Workshop in 2016, Trevor Wooley suggested that the heuristic developed by Keating and the second author~\cite{CK2,CK4,CK5} has an interpretation in terms of the counting of rational points in algebraic varieties that is the subject of Manin's arithmetic stratification conjectures~\cite{browning,fmt,lehmanntanimoto}. Thus, we suspect that the sums \eqref{Sldef} and Conjecture~\ref{conj: stratification} present a stratification of $\mathcal{M}_{A,B}(T,X)$ that has the same interpretation. We may think of the problem of evaluating \eqref{eqn: dirichletpolynomialapprox2} as involving counting solutions (weighted by divisor functions) in the variety
\begin{equation*}
m_1\cdots m_{\ell} -n_1\cdots n_{\ell} =h, \hspace{.25in} |h|\leq H
\end{equation*}
for some parameter $H$. In making the definition \eqref{Sldef}, we are essentially stratifying this variety into the subvarieties
\begin{align*}
m_1N_1 - n_1M_1 & = h_1\\
m_2N_2 - n_2M_2 & = h_2 \\
& \vdots \\
m_{\ell}N_{\ell} - n_{\ell}M_{\ell} & = h_{\ell}
\end{align*}
with $|h_1h_2\cdots h_{\ell}|\leq H$. Thus, as suggested by Wooley, our approach is analogous to counting rational points on high dimensional varieties by stratification and counting points on subvarieties. Note that this stratification introduces some overcounting of solutions. However, we believe that the factor $1/(\ell!)^{2}$ in the definition \eqref{Sldef} accounts for this overcounting. This factor may be explained intuitively by an argument similar to the one in Section~12.2 of \cite{CK5}.

As mentioned earlier, we improve the method in \cite{CK5} by using the integral form \eqref{additivedivisorhypothesis} of the prediction of the delta method. A key observation in our refinement is that the functions $G_E(s,q)$, defined by \eqref{bigGEdef}, that appear in the prediction of the delta method are closely related to the coefficients $I_{C,D}(m)$ in the Dirichlet series expansion
\begin{equation}\label{IABdef}
\frac{\prod_{\gamma\in C}\zeta( s+\gamma)}{ \prod_{\delta\in D}\zeta( s+\delta) } = \sum_{m=1}^{\infty} \frac{I_{C,D}(m)}{m^s} \hspace{.25in} (\re(s)\rightarrow\infty)
\end{equation}
(see Lemma~\ref{GtoI} below). Thus, we are able to adapt the local calculations in \cite{CK5} without much difficulty, as the coefficients $I_{C,D}(m)$ exhibit properties similar to those of the coefficients $\tau_E$ defined by \eqref{taudef}. Another way we refine the approach of \cite{CK5} is in making some of their technical arguments more precise. This includes expressing Fourier transforms of test functions in terms of the gamma function (Lemma~\ref{betalemma}), as has been done in \cite{hamiehng}, \cite{hughesyoung},  and \cite{ng}.

In future work, we aim to adapt the method to other families of $L$-functions and make parts of it more rigorous.

\

\noindent{\bf Acknowledgments.}\, We would like to thank Brad Rodgers for useful comments that improved our exposition. The first author is supported by NSF DMS-1854398 FRG, and the second author is partially supported by a grant from the NSF.

\section{Descending through the recipe}\label{sec: recipe}
We first review the prediction of the CFKRS recipe. We apply Mellin inversion to deduce from the definition \eqref{eqn: dirichletpolynomialapprox} that
\begin{align*}
\mathcal{M}_{A,B}(T,X) = 
& \frac{1}{(2\pi i )^2} \int_{(2)} \int_{(2)} \widetilde{\Upsilon}(\xi) \widetilde{\Upsilon}(\eta) X^{\xi + \eta} \frac{1}{T}\int_0^{\infty} \psi\left( \frac{t}{T}\right)\\
& \times \prod_{\alpha\in A} \zeta(\tfrac{1}{2} +it+\xi +\alpha) \prod_{\beta\in B} \zeta(\tfrac{1}{2} -it+\eta +\beta) \,dt \,d\eta\,d\xi, 
\end{align*}
where $\widetilde{\Upsilon}$ is the Mellin transform of $\Upsilon$, which is defined by
\begin{equation*}
\widetilde{\Upsilon} (s) : = \int_0^\infty \Upsilon(x) x^{s-1}\,dx.
\end{equation*}
We move the lines of integration to $\re(\xi)=\varepsilon$ and $\re(\eta)=\varepsilon$. There are residues from the poles at $\xi=\frac{1}{2}-\alpha-it$, $\alpha\in A$ and $\eta=\frac{1}{2}-\beta+it$, $\beta\in B$. These residues are negligible due to the rapid decay of the Mellin transform $\widetilde{\Upsilon}$ and the fact that $t\asymp T$ since $\psi$ is supported on $[1,2]$. We thus arrive at
\begin{align*}
\mathcal{M}_{A,B}(T,X) = 
& \frac{1}{(2\pi i )^2} \int_{(\varepsilon)} \int_{(\varepsilon)} \widetilde{\Upsilon}(\xi) \widetilde{\Upsilon}(\eta) X^{\xi + \eta} \frac{1}{T}\int_0^{\infty} \psi\left( \frac{t}{T}\right)\\
& \times \prod_{\alpha\in A} \zeta(\tfrac{1}{2} +it+\xi +\alpha) \prod_{\beta\in B} \zeta(\tfrac{1}{2} -it+\eta +\beta) \,dt \,d\eta\,d\xi  \ + \ O_C(T^{-C}), 
\end{align*}
where $C>0$ is arbitrarily large. Applying the recipe in \cite{cfkrs} to the $t$-integral, we conjecture~\eqref{recipe}.

\section{Ascending through convolution sums: applying the delta method}

We apply Mellin inversion and interchange the order of summation to deduce from the definition \eqref{Sldef} of $\mathcal{S}_{\ell}$ that
\begin{align*}
&  \mathcal{S}_{\ell}  = \frac{1}{(\ell !)^2} \sum_{\substack{M_1\cdots M_{\ell} = N_1\cdots N_{\ell} \\ (M_j,N_j)=1 \ \forall j}} \sum_{ \substack{h_1,\dots,h_{\ell} \in \mathbb{Z} \\ h_1\cdots h_{\ell}\neq 0} } \frac{1}{(2\pi i )^2} \int_{(2)} \int_{(2)} \widetilde{\Upsilon}(\xi) \widetilde{\Upsilon}(\eta) X^{\xi + \eta}  \\
& \times  \sum_{\substack{ 1 \leq m_1,\dots,m_{\ell} <\infty \\ 1\leq n_1,\dots,n_{\ell} <\infty  \\ m_jN_j-n_jM_j = h_j \ \forall j }  }  \frac{\tau_{A_1} (m_1)\cdots \tau_{A_{\ell}} (m_{\ell}) \tau_{B_1} (n_1)\cdots \tau_{B_{\ell}} (n_{\ell}) }{(m_1 \cdots m_{\ell} )^{\frac{1}{2}+\xi} (n_1 \cdots n_{\ell} )^{\frac{1}{2}+\eta}} \hat{\psi} \left( \frac{T}{2\pi} \log \frac{m_1 \cdots m_{\ell} }{n_1 \cdots n_{\ell} } \right) \,d\xi\,d\eta,
\end{align*}
where $\hat{\psi}$ is the Fourier transform defined by
\begin{equation}\label{fourierdef}
\hat{\psi}(x) = \int_{-\infty}^{\infty} \psi(t) e^{-2\pi i xt}\,dt.
\end{equation}
Since $M_1\cdots M_{\ell}=N_1\cdots N_{\ell}$, it follows that
\begin{align}
&  \mathcal{S}_{\ell}  = \frac{1}{(\ell !)^2} \sum_{\substack{M_1\cdots M_{\ell} = N_1\cdots N_{\ell} \\ (M_j,N_j)=1 \ \forall j}} \sum_{ \substack{h_1,\dots,h_{\ell} \in \mathbb{Z} \\ h_1\cdots h_{\ell}\neq 0} } \frac{1}{(2\pi i )^2} \int_{(2)} \int_{(2)} \widetilde{\Upsilon}(\xi) \widetilde{\Upsilon}(\eta) X^{\xi + \eta} (N_1\cdots N_{\ell})^{\frac{1}{2}+\xi} (M_1\cdots M_{\ell})^{\frac{1}{2}+\eta}  \label{Sldef2}\\
& \times  \sum_{\substack{ 1 \leq m_1,\dots,m_{\ell} <\infty \\ 1\leq n_1,\dots,n_{\ell} <\infty  \\ m_jN_j-n_jM_j = h_j \ \forall j }  }  \frac{\tau_{A_1} (m_1)\cdots \tau_{A_{\ell}} (m_{\ell}) \tau_{B_1} (n_1)\cdots \tau_{B_{\ell}} (n_{\ell}) }{(m_1 N_1 \cdots m_{\ell}N_{\ell} )^{\frac{1}{2}+\xi} (n_1M_1 \cdots n_{\ell}M_{\ell} )^{\frac{1}{2}+\eta}} \hat{\psi} \left( \frac{T}{2\pi} \log \frac{m_1 N_1 \cdots m_{\ell}N_{\ell} }{n_1 M_1 \cdots n_{\ell}M_{\ell} } \right) \,d\xi\,d\eta. \notag
\end{align}

Now the delta method in \cite{dfi} predicts for a suitable function $f$ that
\begin{align}
& \mathop{\sum\sum}\limits_{\substack{1\leq m,n<\infty \\ mN-nM=h}} \tau_A(m)\tau_B(n) f(mN,nM) \notag\\
& \hspace{.25in} \sim \ \frac{1}{(2\pi i )^2} \oint_{|z|=\varepsilon}\oint_{|w|=\varepsilon} \frac{1}{N^{1+z}M^{1+w}} \prod_{\alpha\in A } \zeta(1+z+\alpha) \prod_{\beta\in B } \zeta(1+w+\beta) \notag\\
& \hspace{.5in}\times \sum_{q=1}^{\infty} \frac{c_q(h) (q,N)^{1+z} (q,M)^{1+w}}{ q^{2+z+w} } G_A\left( 1+z,\frac{q}{(q,N)} \right) G_B\left( 1+w,\frac{q}{(q,M)} \right)  \notag\\
& \hspace{.75in}\times \int_{-\infty}^{\infty} x^z (x-h)^w f(x,x-h) \,dx \,dw\,dz, \label{additivedivisorhypothesis}
\end{align}
where
\begin{equation*}
c_q(h) := \sum_{\substack{a\bmod{q} \\ (a,q)=1}} e^{2\pi i an/q}
\end{equation*}
is the Ramanujan sum and $G_E$ is defined for finite multisets $E$ of complex numbers by
\begin{equation}\label{bigGEdef}
G_{E}(s,q) = \sum_{ d|q} \frac{\mu(d) d^s}{\phi(d)} \sum_{e|d} \frac{\mu(e)}{e^s}  g_{E}\left(s,\frac{eq}{d}\right),
\end{equation}
with $g_{E}(s,n)$ defined by
\begin{equation}\label{smallgedef}
g_{E}(s,n) = \prod_{p|n}  \Bigg\{ \prod_{\gamma\in E} (1-p^{-s-\gamma})   \Bigg\} \sum_{m=0}^{\infty} \frac{\tau_E (p^{m+\ordp(n)}) }{p^{ms}} .
\end{equation}
For details on how to derive this prediction, see Section~1 of \cite{conreygonek}. Forms of this prediction are stated in Section~3 of \cite{CK3}, equation (4) of \cite{CK4}, and Section~6 of \cite{CK5}. We may put these forms into \eqref{additivedivisorhypothesis} by interpreting them as a sum of residues and writing the sum of residues as a contour integral. The left-hand side of \eqref{additivedivisorhypothesis} may be called a \textit{correlation of divisor-sums}. (In the case $M=N=1$, some authors call it a ``shifted convolution sum.'')

We apply the prediction \eqref{additivedivisorhypothesis} to each $m_j,n_j$-sum in \eqref{Sldef2} by taking in \eqref{additivedivisorhypothesis} $A=A_j$, $B=B_j$, $m=m_j$, $n=n_j$, $M=M_j$, $N=N_j$, and
\begin{align*}
f(mN,nM)
& = f(m_1N_1,n_1M_1;m_2N_2,n_2M_2;\dots; m_{\ell}N_{\ell},n_{\ell}M_{\ell}) \\
& =  (m_1 N_1 \cdots m_{\ell}N_{\ell} )^{-\frac{1}{2}-\xi} (n_1M_1 \cdots n_{\ell}M_{\ell} )^{-\frac{1}{2}-\eta} \hat{\psi} \left( \frac{T}{2\pi} \log \frac{m_1 N_1 \cdots m_{\ell}N_{\ell} }{n_1 M_1 \cdots n_{\ell}M_{\ell} } \right).
\end{align*}
We also assume an \textit{independence hypothesis} in the sense that any error terms implied in the conjecture \eqref{additivedivisorhypothesis} do not contribute to the main term in the resulting expression for $\mathcal{S}_{\ell}$. This leads us to predict that
\begin{align}
\mathcal{S}_{\ell} \sim & \frac{1}{(\ell !)^2}  \sum_{\substack{M_1\cdots M_{\ell} = N_1\cdots N_{\ell} \\ (M_j,N_j)=1 \ \forall j}} \sum_{ \substack{h_1,\dots,h_{\ell} \in \mathbb{Z} \\ h_1\cdots h_{\ell}\neq 0} } \frac{1}{(2\pi i )^2} \int_{(2)} \int_{(2)} \widetilde{\Upsilon}(\xi) \widetilde{\Upsilon}(\eta) X^{\xi + \eta}    (N_1\cdots N_{\ell})^{\frac{1}{2}+\xi} \notag\\
& \times (M_1\cdots M_{\ell})^{\frac{1}{2}+\eta}\int_{\max\{0,h_1\}}^{\infty} \cdots \int_{\max\{0,h_{\ell}\}}^{\infty} \hat{\psi} \left( \frac{T}{2\pi} \log \frac{x_1\cdots x_{\ell}}{(x_1-h_1)\cdots (x_{\ell}-h_{\ell})} \right)\notag\\
& \hspace{.25in} \times \prod_{j=1}^{\ell} \Bigg\{ \frac{1}{(2\pi i )^2} \oint_{|z|=\varepsilon}\oint_{|w|=\varepsilon} \frac{ x_j^{-\frac{1}{2}-\xi+z} (x_j-h_j)^{-\frac{1}{2}-\eta +w}}{N_j^{1+z}M_j^{1+w}} \prod_{\alpha\in A_j } \zeta(1+z+\alpha)  \notag\\
& \hspace{.5in}  \times \prod_{\beta\in B_j } \zeta(1+w+\beta)\sum_{q=1}^{\infty} \frac{c_q(h_j) (q,N_j)^{1+z} (q,M_j)^{1+w}}{ q^{2+z+w} } G_{A_j}\left( 1+z,\frac{q}{(q,N_j)} \right)    \notag\\
& \hspace{.75in} \times G_{B_j}\left( 1+w,\frac{q}{(q,M_j)} \right) \,dw\,dz \Bigg\} \,dx_1\cdots\,dx_{\ell}\,d\xi\,d\eta. \label{applydeltamethod}
\end{align}
We let $\epsilon_j=\text{sgn}(h_j)$ and relabel $h_j$ as $\epsilon_jh_j$, where now $h_j$ is positive. Then, we make the change of variables $x_j\mapsto h_jy_j$ for each $j$ to see that \eqref{applydeltamethod} implies
\begin{align}
\mathcal{S}_{\ell} \sim & \frac{1}{(\ell !)^2} \sum_{\substack{M_1\cdots M_{\ell} = N_1\cdots N_{\ell} \\ (M_j,N_j)=1 \ \forall j}} \sum_{\epsilon_1,\dots,\epsilon_{\ell}\in\{1,-1\}} \sum_{ 1\leq h_1,\dots,h_{\ell} <\infty } \frac{1}{(2\pi i )^2} \int_{(2)} \int_{(2)} \widetilde{\Upsilon}(\xi) \widetilde{\Upsilon}(\eta) X^{\xi + \eta}  (N_1\cdots N_{\ell})^{\frac{1}{2}+\xi}   \label{applydeltamethod2}\\
& \times (M_1\cdots M_{\ell})^{\frac{1}{2}+\eta} \int_{\max\{0,\epsilon_1\}}^{\infty} \cdots \int_{\max\{0,\epsilon_{\ell}\}}^{\infty} \hat{\psi} \left( \frac{T}{2\pi} \log \frac{y_1\cdots y_{\ell}}{(y_1-\epsilon_1)\cdots (y_{\ell}-\epsilon_{\ell})} \right)\notag\\
& \hspace{.25in} \times \prod_{j=1}^{\ell} \Bigg\{ \frac{1}{(2\pi i )^2} \oint_{|z|=\varepsilon}\oint_{|w|=\varepsilon} \frac{ h_j^{-\xi+z-\eta+w} y_j^{-\frac{1}{2}-\xi+z} (y_j-\epsilon_j)^{-\frac{1}{2}-\eta +w}}{N_j^{1+z}M_j^{1+w}} \prod_{\alpha\in A_j } \zeta(1+z+\alpha)  \notag\\
& \hspace{.5in}  \times \prod_{\beta\in B_j } \zeta(1+w+\beta) \sum_{q=1}^{\infty} \frac{c_q(h_j) (q,N_j)^{1+z} (q,M_j)^{1+w}}{ q^{2+z+w} } G_{A_j}\left( 1+z,\frac{q}{(q,N_j)} \right)     \notag\\
& \hspace{.75in} \times G_{B_j}\left( 1+w,\frac{q}{(q,M_j)} \right) \,dw\,dz \Bigg\} \,dy_1\cdots\,dy_{\ell}\,d\xi\,d\eta. \notag
\end{align}
We next insert the identity
\begin{equation*}
c_{q }(h_j) = \sum_{\substack{d |q  \\d |h_j}} d \mu \left( \frac{q }{d }\right)
\end{equation*}
for the Ramanujan sum, and then relabel $h_j$ as $h_jd$ and $q$ as $qd$ to find that \eqref{applydeltamethod2} implies
\begin{align}
\mathcal{S}_{\ell} \sim & \frac{1}{(\ell !)^2} \sum_{\substack{M_1\cdots M_{\ell} = N_1\cdots N_{\ell} \\ (M_j,N_j)=1 \ \forall j}} \sum_{\epsilon_1,\dots,\epsilon_{\ell}\in\{1,-1\}}  \frac{1}{(2\pi i )^2} \int_{(2)} \int_{(2)} \widetilde{\Upsilon}(\xi) \widetilde{\Upsilon}(\eta) X^{\xi + \eta}    (N_1\cdots N_{\ell})^{\frac{1}{2}+\xi} \notag\\
& \times (M_1\cdots M_{\ell})^{\frac{1}{2}+\eta} \int_{\max\{0,\epsilon_1\}}^{\infty} \cdots \int_{\max\{0,\epsilon_{\ell}\}}^{\infty} \hat{\psi} \left( \frac{T}{2\pi} \log \frac{y_1\cdots y_{\ell}}{(y_1-\epsilon_1)\cdots (y_{\ell}-\epsilon_{\ell})} \right)\notag\\
& \hspace{.25in} \times \prod_{j=1}^{\ell} \Bigg\{ \frac{1}{(2\pi i )^2} \oint_{|z|=\varepsilon}\oint_{|w|=\varepsilon} \frac{  y_j^{-\frac{1}{2}-\xi+z} (y_j-\epsilon_j)^{-\frac{1}{2}-\eta +w}}{N_j^{1+z}M_j^{1+w}} \zeta(\xi+\eta-z-w)  \notag\\
& \hspace{.5in}  \times \prod_{\alpha\in A_j } \zeta(1+z+\alpha)\prod_{\beta\in B_j } \zeta(1+w+\beta) \sum_{d=1}^{\infty} \frac{1}{d^{1+\xi+\eta}}       \notag\\
& \hspace{.75in} \times \sum_{q=1}^{\infty} \frac{\mu(q) (qd,N_j)^{1+z} (qd,M_j)^{1+w}}{ q^{2+z+w} }G_{A_j}\left( 1+z,\frac{qd}{(qd,N_j)} \right) \notag \\
& \hspace{1in} \times G_{B_j}\left( 1+w,\frac{qd}{(qd,M_j)} \right) \,dw\,dz \Bigg\} \,dy_1\cdots\,dy_{\ell}\,d\xi\,d\eta \notag
\end{align}
because
\begin{equation*}
\sum_{1\leq h_j<\infty} h_j^{-\xi+z-\eta+w} = \zeta(\xi+\eta-z-w)
\end{equation*}
for each $j$.

We next move the $\xi$- and $\eta$-lines to $\re(\xi)=2\varepsilon$ and $\re(\eta)=2\varepsilon$. Doing so traverses the poles of the factors $\zeta(\xi+\eta-z-w)$. However, we expect the residues of the integrand at these poles to be negligible because of the presence of the factor $\chi(1+z+w-\xi-\eta)$ in the consequence~\eqref{lemma3.1consequence} of Lemma~\ref{betalemma} below. This factor is zero at the pole of $\zeta(\xi+\eta-z-w)$. We then insert the definition \eqref{fourierdef} of $\hat{\psi}$ and interchange the order of summation to arrive at the prediction
\begin{align}
\mathcal{S}_{\ell} \sim & \frac{1}{(\ell!)^2 (2\pi i )^2} \int_{(2\varepsilon)} \int_{(2\varepsilon)} \widetilde{\Upsilon}(\xi) \widetilde{\Upsilon}(\eta) \frac{X^{\xi + \eta}}{T} \int_0^{\infty} \psi\left( \frac{t}{T}\right) \frac{1}{(2\pi i)^{2\ell}} \oint_{|z_1|=\varepsilon} \cdots \oint_{|z_{\ell}|=\varepsilon}   \notag\\
& \times \oint_{|w_1|=\varepsilon} \cdots \oint_{|w_{\ell}|=\varepsilon}\sum_{\substack{M_1\cdots M_{\ell} = N_1\cdots N_{\ell} \\ (M_j,N_j)=1 \ \forall j}} \prod_{j=1}^{\ell} \Bigg\{ \zeta(\xi+\eta-z_j-w_j) N_j^{-\frac{1}{2}+\xi-z_j} M_j^{-\frac{1}{2}+\eta-w_j} \notag\\
& \hspace{.25in} \times\sum_{\epsilon_j=\pm 1 }       \int_{\max\{0,\epsilon_j\}}^{\infty}  \left( \frac{y_j-\epsilon_j}{y_j} \right)^{it}       y_j^{-\frac{1}{2}-\xi+z_j} (y_j-\epsilon_j)^{-\frac{1}{2}-\eta +w_j} \prod_{\alpha\in A_j } \zeta(1+z_j+\alpha)  \notag\\
& \hspace{.5in}  \times \prod_{\beta\in B_j } \zeta(1+w_j+\beta) \sum_{d=1}^{\infty} \frac{1}{d^{1+\xi+\eta}} \sum_{q=1}^{\infty} \frac{\mu(q) (qd,N_j)^{1+z_j} (qd,M_j)^{1+w_j}}{ q^{2+z_j+w_j} }      \notag\\
& \hspace{.75in} \times G_{A_j}\left( 1+z_j,\frac{qd}{(qd,N_j)} \right)G_{B_j}\left( 1+w_j,\frac{qd}{(qd,M_j)} \right) \,dy_j \,dw_{j} \,dz_j \Bigg\}\,d\xi\,d\eta. \label{applydeltamethod3}
\end{align}
We may now evaluate each $y_j$-integral through the following lemma.

\begin{lemma}\label{betalemma}
If $a,b$ are complex numbers with $0<\re(a),\re(b)<\frac{1}{2}$, then
\begin{align*}
& \sum_{\epsilon=\pm 1}\int_{\max\{0,\epsilon\}}^{\infty}y^{a-1} (y-\epsilon)^{b-1}\,dy \\
& = \chi(a+b) \chi(1-a)\chi(1-b)\left(\frac{1+\tan(\frac{\pi}{2} a)\tan(\frac{\pi}{2} b) }{2} \right),
\end{align*}
where $\chi(s)$ is the factor in the functional equation $\zeta(s)=\chi(s)\zeta(1-s)$.
\end{lemma}
\begin{proof}
For brevity, define $\mathcal{I}(a,b)$ by
\begin{equation*}
\mathcal{I}(a,b): = \sum_{\epsilon=\pm 1}\int_{\max\{0,\epsilon\}}^{\infty}y^{a-1} (y-\epsilon)^{b-1}\,dy.
\end{equation*}
We have
\begin{align*}
\mathcal{I}(a,b) = B(a,1-a-b)+B(b,1-a-b),
\end{align*}
where $B(\cdot,\cdot)$ is the beta function. Writing the beta function in terms of the gamma function, we deduce that
\begin{equation*}
\mathcal{I}(a,b)
= \Bigg( \frac{\Gamma(a) }{\Gamma(1-b)} + \frac{\Gamma(b) }{\Gamma(1-a)}\Bigg) \Gamma(1-a-b).
\end{equation*}
Since $\Gamma(s)\Gamma(1-s)=\pi/\sin(\pi s)$, it follows that
\begin{equation*}
\mathcal{I}(a,b)
= \Bigg( \frac{1 }{\sin \pi a} + \frac{1 }{\sin \pi b}\Bigg)\frac{\pi  \Gamma(1-a-b)}{\Gamma(1-a)\Gamma(1-b)}.
\end{equation*}
Now the identity $\sin A+\sin B=2\sin(\frac{A}{2}+\frac{B}{2})\cos(\frac{A}{2}-\frac{B}{2})$ implies
$$
\frac{1 }{\sin \pi a} + \frac{1 }{\sin \pi b} = \frac{2\sin ( \frac{\pi}{2}(a+b)) \cos (\frac{\pi}{2}(a-b))}{\sin (\pi a) \sin (\pi b)}.
$$
Thus
\begin{align*}
\mathcal{I}(a,b) = \frac{\chi(a+b)\cos (\frac{\pi}{2}(a-b)) }{2^{a+b-1}\pi^{a+b-2}\Gamma(1-a)\Gamma(1-b)\sin (\pi a) \sin (\pi b) } ,
\end{align*}
where we have also used the definition $\chi(s)=2^s\pi^{s-1} \sin (\frac{\pi}{2} s) \Gamma(1-s)$. The double-angle formulae then imply that
\begin{equation*}
\mathcal{I}(a,b) = \frac{\chi(a+b)\cos (\frac{\pi}{2}(a-b)) }{2\chi(a)\chi(b) \cos (\frac{\pi}{2} a) \cos (\frac{\pi}{2} b)}.
\end{equation*}
The lemma now follows from this, the identity $\chi(s)\chi(1-s)=1$, and the fact that
$$
 \frac{ \cos (\frac{\pi}{2}(a-b)) }{ \cos (\frac{\pi}{2} a) \cos (\frac{\pi}{2} b)} = 1+\tan(\tfrac{\pi}{2} a)\tan(\tfrac{\pi}{2} b).
$$
\end{proof}

If $\re(\xi)=\re(\eta)=2\varepsilon$, $|z_j|,|w_j|=\varepsilon$, and $t$ is real, then it follows from Lemma~\ref{betalemma} with $a=\frac{1}{2}-\xi+z_j-it$ and $b=\frac{1}{2}-\eta+w_j+it$ that
\begin{align}
& \sum_{\epsilon_j=\pm 1 }       \int_{\max\{0,\epsilon_j\}}^{\infty}  \left( \frac{y_j-\epsilon_j}{y_j} \right)^{it} y_j^{-\frac{1}{2}-\xi+z_j} (y_j-\epsilon_j)^{-\frac{1}{2}-\eta +w_j} \,dy_j \notag\\
& \hspace{.25in} = \chi(1+z_j+w_j-\xi-\eta)\chi(\tfrac{1}{2}+\xi-z_j+it)\chi (\tfrac{1}{2}+\eta-w_j-it) \notag \\
& \hspace{.5in} \times \left(\frac{1+\tan(\frac{\pi}{2} (\tfrac{1}{2}-\xi+z_j-it))\tan(\frac{\pi}{2} (\tfrac{1}{2}-\eta+w_j+it)) }{2} \right). \label{lemma3.1consequence}
\end{align}
If it also holds that $T\leq t\leq 2T$ and $|\text{Im}(\xi)|,|\text{Im}(\eta)|\leq T/2$, say, then
$$
\tan(\tfrac{\pi}{2} (\tfrac{1}{2}-\xi+z_j-it))\tan(\tfrac{\pi}{2} (\tfrac{1}{2}-\eta+w_j+it)) = 1 +O\big( e^{-T/4} \big).
$$
From this, the functional equation
$$
\zeta(\xi+\eta-z_j-w_j) \chi(1+z_j+w_j-\xi-\eta) = \zeta(1+z_j+w_j-\xi-\eta),
$$
and \eqref{applydeltamethod3}, we arrive at the prediction
\begin{align}
\mathcal{S}_{\ell} \sim & \frac{1}{(\ell!)^2 (2\pi i )^2} \int_{(2\varepsilon)} \int_{(2\varepsilon)} \widetilde{\Upsilon}(\xi) \widetilde{\Upsilon}(\eta) \frac{X^{\xi + \eta}}{T} \int_0^{\infty} \psi\left( \frac{t}{T}\right) \frac{1}{(2\pi i)^{2\ell}} \oint_{|z_1|=\varepsilon} \cdots \oint_{|z_{\ell}|=\varepsilon}   \notag\\
& \times \oint_{|w_1|=\varepsilon} \cdots \oint_{|w_{\ell}|=\varepsilon} \sum_{\substack{M_1\cdots M_{\ell} = N_1\cdots N_{\ell} \\ (M_j,N_j)=1 \ \forall j}} \prod_{j=1}^{\ell} \Bigg\{ \zeta(1+z_j+w_j-\xi-\eta) \chi (\tfrac{1}{2} +\xi-z_j+it) \notag\\
& \hspace{.25in} \times  \chi (\tfrac{1}{2} +\eta-w_j-it) \prod_{\alpha\in A_j } \zeta(1+z_j+\alpha)  \prod_{\beta\in B_j } \zeta(1+w_j+\beta) \notag\\
& \hspace{.5in}  \times  N_j^{-\frac{1}{2}+\xi-z_j} M_j^{-\frac{1}{2}+\eta-w_j} \sum_{d=1}^{\infty} \frac{1}{d^{1+\xi+\eta}} \sum_{q=1}^{\infty} \frac{\mu(q) (qd,N_j)^{1+z_j} (qd,M_j)^{1+w_j}}{ q^{2+z_j+w_j} }      \notag\\
& \hspace{.75in} \times G_{A_j}\left( 1+z_j,\frac{qd}{(qd,N_j)} \right)G_{B_j}\left( 1+w_j,\frac{qd}{(qd,M_j)} \right) \,dw_{j} \,dz_j \Bigg\}\,d\xi\,d\eta. \label{applydeltamethod4}
\end{align}

\section{Local calculations}

Our next task is to evaluate the sum over the $M_j,N_j$ in \eqref{applydeltamethod4}. By multiplicativity, we may formally write this sum as an Euler product; see \eqref{sumaseulerproduct} below. Our main result for this section (Theorem~\ref{euler}) is an expression for the local factor of this Euler product in terms of the coefficients $I_{C,D}(m)$, which are defined by \eqref{IABdef} for finite multisets $C,D$ of complex numbers. We first prove some basic properties of these coefficients.
\begin{lemma}\label{IABlemma}
Let $C$ and $D$ be finite multisets of complex numbers, let $s$ be a complex number, and let $p$ be a prime.
\begin{enumerate}
\item[\upshape{(i)}] If $m$ is a positive integer, then
\begin{equation*}
m^{-s} I_{C,D}(m) = I_{C_s,D_s}(m),
\end{equation*}
where, for every multiset $E$, $E_s$ denotes the multiset $\{\gamma+s : \gamma\in E\}$.

\item[\upshape{(ii)}] If $r$ is a nonnegative integer, then
\begin{equation*}
I_{C\cup \{s\},D}(p^r) = I_{C,D}(p^r) +p^{-s} I_{C\cup \{s\},D}(p^{r-1}),
\end{equation*}
where the last term is to be interpreted as $0$ if $r=0$.

\item[\upshape{(iii)}] If $R$ and $M$ are nonnegative integers, then
\begin{equation*}
\sum_{r=0}^R I_{C,D}(p^{r+M})= I_{C\cup \{0\},D}(p^{R+M}) - I_{C\cup \{0\},D}(p^{M-1}),
\end{equation*}
where the last term is to be interpreted as $0$ if $M=0$.

\item[\upshape{(iv)}] We have
\begin{equation*}
\sum_{k=0}^{\infty}I_{C,\{-s\}}(p^k)p^{- k(1+s)} = \left(1-\frac{1}{p}\right) \prod_{\gamma\in C}\left( 1-\frac{1}{p^{1+s+\gamma}}\right)^{-1}
\end{equation*}
whenever the left-hand side converges absolutely.
\end{enumerate}
\end{lemma}
\begin{proof}
If $C=\{\gamma_1,\dots,\gamma_h\}$ and $D=\{\delta_1,\dots,\delta_{\ell}\}$, then the definition \eqref{IABdef} implies
\begin{equation}\label{IABdef2}
I_{C,D} (m) = \sum_{m_1\cdots m_h n_1\cdots n_{\ell} = m} m_1^{-\gamma_1} \cdots m_h^{-\gamma_h} \mu(n_1) n_1^{-\delta_1} \cdots \mu(n_{\ell}) n_{\ell}^{-\delta_{\ell}},
\end{equation}
and (i) immediately follows.

To prove (ii), use \eqref{IABdef2} to write
\begin{equation}\label{IABlemma(ii)}
I_{C\cup \{s\},D} (p^r) = \sum_{\nu+m_1+\cdots +m_h +n_1+\cdots +n_{\ell} = r} p^{-\nu s} p^{-m_1\gamma_1} \cdots p^{-m_h\gamma_h} \mu(p^{n_1}) p^{-n_1\delta_1} \cdots \mu(p^{n_{\ell}}) p^{-n_{\ell}\delta_{\ell}},
\end{equation}
where $\nu$ and the $m_i$'s and $n_i$'s run through nonnegative integers. By \eqref{IABdef2}, we see that $I_{C,D}(p^r)$ equals the sum of the terms on the right-hand side of \eqref{IABlemma(ii)} that have $\nu=0$, while the sum of the terms with $\nu \geq 1$ equals $p^{-s} I_{C\cup \{s\},D}(p^{r-1})$. This proves (ii). 

Next, to show (iii), we take $s=0$ in (ii) to deduce that
$$
I_{C,D}(p^{r+M}) =I_{C\cup\{0\},D} (p^{r+M})  - I_{C\cup\{0\},D} (p^{r+M-1}).  
$$
Summing both sides from $r=0$ to $r=R$ gives (iii).

Finally, (iv) follows immediately from the definition \eqref{IABdef} and the Euler product expansion of zeta.

\end{proof}

We next prove a generalization of Lemma~2 of \cite{CK5}. To state it, we define
\begin{align}
\Sigma(A,B,z,w,M,N;p) : =
& \sum_{q=0}^1 \sum_{d=0}^{\infty}\sum_{j=0}^{\infty}\sum_{k=0}^{\infty} (-1)^q I_{A,\{-z\}}\Big(p^{j+q+d-\min\{q+d,N\}}\Big)  \notag\\
& \hspace{.25in} \times I_{B,\{-w\}} \Big(p^{k+q+d-\min\{q+d,M\}}\Big) \notag \\
& \hspace{.5in} \times p^{-d-j(1+z)-k(1+w)-q(2+z+w)+\min\{q+d,N\}(1+z) +\min\{q+d,M\}(1+w)}.  \label{SigmaABdef}
\end{align}

\begin{lemma}\label{lemma2CKV}
Let $\epsilon>0$ be arbitrarily small, and let $p$ be a prime. Suppose that $A$ and $B$ are finite multisets of complex numbers and $z$ and $w$ are complex numbers such that $|\re(\gamma)|\leq \epsilon$ for all $\gamma\in A\cup B\cup\{z,w\}$. If $M$ and $N$ are nonnegative integers such that $\min\{M,N\}=0$, then
\begin{align*}
\Sigma(A,B,z,w,M,N;p) = p^{Mw+Nz} \left(1-\frac{1}{p^{1+w+z}}\right) \sum_{r=0}^{\infty} \frac{I_{A\cup \{w\},\{-z\}}(p^{r+M}) I_{B\cup \{z\},\{-w\}}(p^{r+N})}{ p^r}.
\end{align*}
\end{lemma}

\begin{proof}
First, observe that \eqref{IABdef2} and the divisor bound imply that if the elements of $C$ and $D$ each have real part $\geq -c$, then
\begin{equation}\label{divisorbound}
I_{C,D}(m)\ll_{\varepsilon} m^{c + \varepsilon}
\end{equation}
for arbitarily small $\varepsilon>0$. From this and the assumption that $|\re(\gamma)|\leq \epsilon$ for all $\gamma\in A\cup B\cup\{z,w\}$, we deduce the absolute convergence of the right-hand side of \eqref{SigmaABdef}. Thus, the sum $\Sigma(A,B,z,w,M,N;p)$ is well-defined.

To prove Lemma~\ref{lemma2CKV}, we may assume without loss of generality that $N=0$, since the case with $M=0$ will follow by symmetry. If $N=0$, then the definition \eqref{SigmaABdef} specializes to
\begin{align*}
\Sigma(A,B,z,w,M,0;p) = \sum_{q=0}^1
&  \sum_{d=0}^{\infty}\sum_{j=0}^{\infty}\sum_{k=0}^{\infty} (-1)^q I_{A,\{-z\}}(p^{j+q+d}) I_{B,\{-w\}} (p^{k+q+d-\min\{q+d,M\}})\\
& \times  p^{-d-j(1+z)-k(1+w)-q(2+z+w) +\min\{q+d,M\}(1+w)}.
\end{align*}
Split this into
\begin{equation}\label{Sigma+-}
\Sigma(A,B,z,w,M,0;p) = \Sigma^- + \Sigma^+,
\end{equation}
where $\Sigma^-$ is the sum of the terms with $d<M$, and $\Sigma^+$ is the sum of the terms with $d\geq M$. 

We first evaluate $\Sigma^-$. If $d<M$ and $q\in\{0,1\}$, then $q+d-\min\{q+d,M\}=0$ and $d+q(2+z+w)-\min\{q+d,M\}(1+w)=-dw+q(1+z)$. We use these and then carry out the summation over $q$ to write
\begin{align*}
\Sigma^- =
&   \sum_{d<M}\sum_{j=0}^{\infty}\sum_{k=0}^{\infty}   I_{A,\{-z\}}(p^{j +d}) I_{B,\{-w\}}(p^k)  p^{-j(1+z)-k(1+w)  +dw  }\\
& -  \sum_{d<M}\sum_{j=0}^{\infty}\sum_{k=0}^{\infty} I_{A,\{-z\}}(p^{j+1+d}) I_{B,\{-w\}}(p^k)  p^{-(j+1)(1+z)-k(1+w)  +dw  }.
\end{align*}
Factor out the $k$-sum to deduce that
\begin{align*}
\Sigma^- = \sum_{k=0}^{\infty}I_{B,\{-w\}}(p^k)p^{- k(1+w)} \Bigg\{
& \sum_{d<M}\sum_{j=0}^{\infty}   I_{A,\{-z\}}(p^{j +d})   p^{-j(1+z)   +dw  }\\
& -  \sum_{d<M}\sum_{j=0}^{\infty}  I_{A,\{-z\}}(p^{j+1+d})   p^{-(j+1)(1+z)   +dw  } \Bigg\}.
\end{align*}
The $j$-sums telescope, while we may evaluate the $k$-sum using Lemma~\ref{IABlemma}(iv). Hence
\begin{equation*}
\Sigma^- = \left(1-\frac{1}{p}\right) \prod_{\beta\in B}\left( 1-\frac{1}{p^{1+w+\beta}}\right)^{-1} \sum_{d<M}   I_{A,\{-z\}}(p^{ d})   p^{ dw  }.
\end{equation*}
From (i) and (iii) of Lemma~\ref{IABlemma}, we see that
\begin{equation*}
\begin{split}
\sum_{d<M}   I_{A,\{-z\}}(p^{ d})   p^{ dw  }
& = \sum_{d<M}   I_{A_{-w},\{-z-w\}}(p^{ d}) \\
& = I_{A_{-w}\cup\{0\},\{-z-w\}}(p^{ M-1})  = p^{ (M-1)w  } I_{A\cup\{w\},\{-z\}}(p^{ M-1}).
\end{split}
\end{equation*}
Thus
\begin{equation}\label{Sigma-}
\Sigma^- =  \left(1-\frac{1}{p}\right)p^{ (M-1)w  } I_{A\cup\{w\},\{-z\}}(p^{ M-1})   \prod_{\beta\in B}\left( 1-\frac{1}{p^{1+w+\beta}}\right)^{-1}.
\end{equation}

Having evaluated $\Sigma^-$, we next consider the sum $\Sigma^+$ defined in \eqref{Sigma+-}. If $d\geq M$ and $q\geq 0$, then $\min\{q+d,M\}=M$. We use this and then carry out the summation over $q$ to deduce that
\begin{align*}
\Sigma^+ =
& \sum_{d\geq M}\sum_{j=0}^{\infty}\sum_{k=0}^{\infty} I_{A,\{-z\}}(p^{j+d}) I_{B,\{-w\}}(p^{k+d-M})  p^{-d-j(1+z)-k(1+w) +M(1+w)} \\
& - \sum_{d\geq M}\sum_{j=0}^{\infty}\sum_{k=0}^{\infty}   I_{A,\{-z\}}(p^{j+1+d}) I_{B,\{-w\}}(p^{k+1+d-M})  p^{-d-(j+1)(1+z)-(k+1)(1+w) +M(1+w)}.
\end{align*}
We relabel $d$ as $d+M$ to see that
\begin{align*}
\Sigma^+ =
& p^{Mw}\sum_{d=0}^{\infty}\sum_{j=0}^{\infty}\sum_{k=0}^{\infty} I_{A,\{-z\}}(p^{j+d+M}) I_{B,\{-w\}}(p^{k+d})   p^{-d-j(1+z)-k(1+w)} \\
& - p^{Mw}\sum_{d=0}^{\infty}\sum_{j=0}^{\infty}\sum_{k=0}^{\infty}   I_{A,\{-z\}}(p^{j+1+d+M}) I_{B,\{-w\}}(p^{k+1+d})   p^{-d-(j+1)(1+z)-(k+1)(1+w)  }.
\end{align*}
We add and subtract an extra sum and arrive at
\begin{align*}
\Sigma^+ =
& p^{Mw}\sum_{d=0}^{\infty}\sum_{j=0}^{\infty}\sum_{k=0}^{\infty} I_{A,\{-z\}}(p^{j+d+M}) I_{B,\{-w\}}(p^{k+d})   p^{-d-j(1+z)-k(1+w)} \\
& - p^{Mw}\sum_{d=0}^{\infty}\sum_{j=0}^{\infty}\sum_{k=0}^{\infty} I_{A,\{-z\}}(p^{j+1+d+M}) I_{B,\{-w\}}(p^{k+d})   p^{-d-(j+1)(1+z)-k(1+w)} \\
& + p^{Mw}\sum_{d=0}^{\infty}\sum_{j=0}^{\infty}\sum_{k=0}^{\infty} I_{A,\{-z\}}(p^{j+1+d+M}) I_{B,\{-w\}}(p^{k+d})   p^{-d-(j+1)(1+z)-k(1+w)} \\
& - p^{Mw}\sum_{d=0}^{\infty}\sum_{j=0}^{\infty}\sum_{k=0}^{\infty}   I_{A,\{-z\}}(p^{j+1+d+M}) I_{B,\{-w\}}(p^{k+1+d})   p^{-d-(j+1)(1+z)-(k+1)(1+w)  }.
\end{align*}
We may combine the first two sums on the right-hand side and see that the $j$-sums telescope to leave only the $j=0$ terms of the first sum. Similarly, we may combine the third and fourth sums and see that the $k$-sums telescope to leave only the $k=0$ terms of the third sum. Therefore
\begin{align*}
\Sigma^+ =
& p^{Mw}\sum_{d=0}^{\infty}\sum_{k=0}^{\infty} I_{A,\{-z\}}(p^{d+M}) I_{B,\{-w\}}(p^{k+d})   p^{-d-k(1+w)}  \\
& + p^{Mw}\sum_{d=0}^{\infty}\sum_{j=0}^{\infty}  I_{A,\{-z\}}(p^{j+1+d+M}) I_{B,\{-w\}}(p^{d})   p^{-d-(j+1)(1+z)}.
\end{align*}
In the last sum, relabel $j+1$ as $j$ and then add and subtract the $j=0$ terms to deduce that
\begin{equation}\label{Sigma+}
\Sigma^+ = p^{Mw} \Big( \Sigma^+_1 + \Sigma^+_2 - \Sigma^+_3 \Big),
\end{equation}
where $\Sigma^+_1$, $\Sigma^+_2$, and $\Sigma^+_3$ are defined by
\begin{equation*}
\Sigma^+_1 = \sum_{d=0}^{\infty}\sum_{k=0}^{\infty} I_{A,\{-z\}}(p^{d+M}) I_{B,\{-w\}}(p^{k+d})   p^{-d-k(1+w)},
\end{equation*}
\begin{equation}\label{Sigma2+}
\Sigma^+_2 = \sum_{d=0}^{\infty}\sum_{j=0}^{\infty}  I_{A,\{-z\}}(p^{j+d+M}) I_{B,\{-w\}}(p^{d})   p^{-d-j(1+z)},
\end{equation}
and
\begin{equation}\label{Sigma3+}
\Sigma^+_3 = \sum_{d=0}^{\infty}   I_{A,\{-z\}}(p^{d+M}) I_{B,\{-w\}}(p^{d})   p^{-d},
\end{equation}
respectively.

We next evaluate each of $\Sigma^+_1$ and $\Sigma^+_2$. We first consider $\Sigma^+_1$. Group together terms that have the same $k+d$ to write
\begin{align*}
\Sigma^+_1 =  \sum_{r=0}^{\infty} I_{B,\{-w\}}(p^{r}) p^{-r} \sum_{k+d=r}  I_{A,\{-z\}}(p^{d+M})p^{-kw}.	
\end{align*}
In this expression, we have $kw=rw-dw=rw-(d+M)w+Mw$. Thus
\begin{align*}
\Sigma^+_1 =  p^{-Mw}\sum_{r=0}^{\infty}  I_{B,\{-w\}}(p^{r}) p^{-r(1+w)} \sum_{k+d=r}  I_{A,\{-z\}}(p^{d+M}) p^{(d+M)w}.
\end{align*}
By (i) and (iii) of Lemma~\ref{IABlemma}, we have
\begin{align*}
\sum_{k+d=r}  I_{A,\{-z\}}(p^{d+M}) p^{(d+M)w}
& = \sum_{k+d=r}  I_{A_{-w},\{-z-w\}}(p^{d+M})\\
& =  I_{A_{-w}\cup \{0\},\{-z-w\}}(p^{r+M}) - I_{A_{-w}\cup \{0\},\{-z-w\}}(p^{M-1}) \\
& = p^{(r+M)w} I_{A\cup \{w\}, \{ -z\}}(p^{r+M}) - p^{(M-1)w} I_{A\cup \{w\}, \{ -z\}}(p^{M-1}).
\end{align*}
Hence
\begin{align*}
\Sigma^+_1 =  \sum_{r=0}^{\infty} I_{B,\{-w\}}(p^r) I_{A\cup \{w\}, \{ -z\}}(p^{r+M}) p^{-r} -p^{-w}I_{A\cup \{w\}, \{ -z\}}(p^{M-1})\sum_{r=0}^{\infty} I_{B,\{-w\}}(p^r) p^{-r(1+w)}.
\end{align*}
From this and Lemma~\ref{IABlemma}(iv), it follows that
\begin{align}
\Sigma^+_1 =  \sum_{r=0}^{\infty}
& I_{B,\{-w\}}(p^r) I_{A\cup \{w\}, \{ -z\}}(p^{r+M}) p^{-r} \notag\\
& -p^{-w} \left(1-\frac{1}{p}\right) I_{A\cup \{w\}, \{ -z\}}(p^{M-1}) \prod_{\beta\in B} \left(1-\frac{1}{p^{1+w+\beta}}\right)^{-1}. \label{Sigma+1}
\end{align}

Having evaluated $\Sigma_1^+$, we next turn our attention to the sum $\Sigma_2^+$ defined by \eqref{Sigma2+}. Gather terms with the same $j+d$ to write
\begin{equation*}
\Sigma_2^+ = \sum_{r=0}^{\infty} I_{A, \{ -z\}}(p^{r+M}) p^{-r} \sum_{j+d=r}  I_{B,\{-w\}}(p^{d})   p^{-jz }.
\end{equation*}
In this sum, we have $jz=rz-dz$. Thus
\begin{equation*}
\Sigma_2^+ = \sum_{r=0}^{\infty} I_{A, \{ -z\}}(p^{r+M}) p^{-r(1+z)} \sum_{j+d=r}  I_{B,\{-w\}}(p^{d})   p^{dz }.
\end{equation*}
From (i) and (iii) of Lemma~\ref{IABlemma}, we see that
\begin{align*}
\sum_{j+d=r} I_{B,\{-w\}}(p^{d})   p^{dz }
=\sum_{j+d=r} I_{B_{-z},\{-w-z\}}(p^{d}) 
=I_{B_{-z}\cup\{0\},\{-w-z\}}(p^{r}) =p^{rz} I_{B \cup\{z\},\{-w\}}(p^{r}).
\end{align*}
Thus
\begin{equation}\label{Sigma+2}
\Sigma_2^+ = \sum_{r=0}^{\infty} I_{A, \{ -z\}}(p^{r+M}) I_{B \cup\{z\},\{-w\}}(p^{r})p^{-r} .
\end{equation}

Combining now \eqref{Sigma+}, \eqref{Sigma3+}, \eqref{Sigma+1}, and \eqref{Sigma+2}, we arrive at
\begin{align*}
\Sigma^+ = & -p^{(M-1)w} \left(1-\frac{1}{p}\right) I_{A\cup \{w\}, \{ -z\}}(p^{M-1})\prod_{\beta\in B} \left(1-\frac{1}{p^{1+w+\beta}}\right)^{-1} \\
& + p^{Mw} \sum_{r=0}^{\infty} I_{A\cup \{w\}, \{ -z\}}(p^{r+M}) I_{B,\{-w\}}(p^{r}) p^{-r}  +  p^{Mw} \sum_{r=0}^{\infty} I_{A , \{ -z\}}(p^{r+M}) I_{B \cup\{z\},\{-w\}}(p^{r})p^{-r} \\
& - p^{Mw} \sum_{r=0}^{\infty} I_{A , \{ -z\}}(p^{r+M}) I_{B ,\{-w\}}(p^{r})p^{-r}.
\end{align*}
From this, \eqref{Sigma+-}, and \eqref{Sigma-}, we arrive at
\begin{align}
& \Sigma(A,B,z,w,M,0;p) = p^{Mw} \sum_{r=0}^{\infty}I_{A\cup \{w\}, \{ -z\}}(p^{r+M}) I_{B ,\{-w\}}(p^{r})p^{-r} \notag\\
&  +  p^{Mw} \sum_{r=0}^{\infty} I_{A , \{ -z\}}(p^{r+M})  I_{B\cup\{z\} ,\{-w\}}(p^{r})p^{-r}  - p^{Mw} \sum_{r=0}^{\infty} I_{A , \{ -z\}}(p^{r+M}) I_{B ,\{-w\}}(p^{r}) p^{-r}. \label{Sigma0}
\end{align}

To write this in a more concise form, we apply the trick on page~746 of \cite{CK3}, as follows. By Lemma~\ref{IABlemma}(ii), it holds that
\begin{align*}
& I_{A\cup \{w\}, \{ -z\}}(p^{r+M}) I_{B ,\{-w\}}(p^{r}) + I_{A , \{ -z\}}(p^{r+M}) I_{B\cup\{z\} ,\{-w\}}(p^{r}) - I_{A , \{ -z\}}(p^{r+M}) I_{B ,\{-w\}}(p^{r}) \\
& = (a_1+a_2) b_1 + a_1(b_1+b_2) -a_1b_1,
\end{align*}
where
\begin{align*}
a_1 & = I_{A , \{ -z\}}(p^{r+M}) \\
a_2 & = p^{-w} I_{A\cup \{w\}, \{ -z\}}(p^{r+M-1})\\
b_1 & = I_{B ,\{-w\}}(p^{r}) \\
b_2 & = p^{-z} I_{B\cup\{z\} ,\{-w\}}(p^{r-1}).
\end{align*}
Cancel the term $-a_1b_1$, then add and subtract $a_2b_2$ to deduce that
\begin{align*}
& I_{A\cup \{w\}, \{ -z\}}(p^{r+M})I_{B ,\{-w\}}(p^{r}) + I_{A , \{ -z\}}(p^{r+M})I_{B\cup\{z\} ,\{-w\}}(p^{r}) - I_{A , \{ -z\}}(p^{r+M}) I_{B ,\{-w\}}(p^{r}) \\
& = (a_1+a_2) b_1 + a_1b_2 +a_2b_2 -a_2b_2 \\
& = (a_1+a_2) (b_1 +b_2) -a_2b_2.
\end{align*}
From this, the definitions of $a_1,a_2,b_1,b_2$, and Lemma~\ref{IABlemma}(ii), we arrive at
\begin{align*}
&  I_{A\cup \{w\}, \{ -z\}}(p^{r+M})I_{B ,\{-w\}}(p^{r}) + I_{A , \{ -z\}}(p^{r+M})I_{B\cup\{z\} ,\{-w\}}(p^{r}) - I_{A , \{ -z\}}(p^{r+M}) I_{B ,\{-w\}}(p^{r}) \\
& =   I_{A\cup \{w\}, \{ -z\}}(p^{r+M}) I_{B\cup\{z\} ,\{-w\}}(p^{r}) - p^{-w-z} I_{A\cup \{w\}, \{ -z\}}(p^{r+M-1}) I_{B\cup\{z\} ,\{-w\}}(p^{r-1}).
\end{align*}
This and \eqref{Sigma0} imply
\begin{align*}
\Sigma(A,B,z,w,M,0;p) = & p^{Mw} \sum_{r=0}^{\infty}I_{A\cup \{w\}, \{ -z\}}(p^{r+M}) I_{B\cup\{z\} ,\{-w\}}(p^{r}) p^{-r} \\
&  - p^{Mw} p^{-w-z}\sum_{r=0}^{\infty}I_{A\cup \{w\}, \{ -z\}}(p^{r+M-1}) I_{B\cup\{z\} ,\{-w\}}(p^{r-1}) p^{-r}.
\end{align*}
The $r=0$ term of the last sum is zero by the convention mentioned in Lemma~\ref{IABlemma}(ii). Hence, we may relabel $r-1$ in this last sum as $r$ to deduce that
\begin{equation}\label{Sigma00}
\Sigma(A,B,z,w,M,0;p) = p^{Mw} \left(1-\frac{1}{p^{1+w+z}}\right) \sum_{r=0}^{\infty} I_{A\cup \{w\}, \{ -z\}}(p^{r+M}) I_{B\cup\{z\} ,\{-w\}}(p^{r}) p^{-r}.
\end{equation}
This proves Lemma~\ref{lemma2CKV} with the additional assumption that $N=0$. Now the definition \eqref{SigmaABdef} of $\Sigma$ implies that
\begin{equation*}
\Sigma(A,B,z,w,M,N;p) = \Sigma(B,A,w,z,N,M;p).
\end{equation*}
It follows from this and \eqref{Sigma00} that
\begin{align*}
\Sigma(A,B,z,w,0,N;p) 
& = \Sigma(B,A,w,z,N,0;p) \\
& = p^{Nz} \left(1-\frac{1}{p^{1+w+z}}\right) \sum_{r=0}^{\infty} I_{A\cup \{w\}, \{ -z\}}(p^{r}) I_{B\cup\{z\} ,\{-w\}}(p^{r+N}) p^{-r}
\end{align*}
This proves Lemma~\ref{lemma2CKV} with the additional assumption that $M=0$. Since $\min\{M,N\}=0$ means either $M=0$ or $N=0$, the proof of Lemma~\ref{lemma2CKV} is complete.
\end{proof}

In order to use Lemma~\ref{lemma2CKV} to evaluate the sum in \eqref{applydeltamethod4}, we need to relate the function $G$, defined by \eqref{bigGEdef}, with the function $I$, which is defined by \eqref{IABdef}.

\begin{lemma}\label{GtoI}
Let $p$ be a prime, and let $r$ be a nonnegative integer. Suppose that $E$ is a finite multiset of complex numbers and $s$ is a complex number such that $\re(s+\gamma)>0$ for every $\gamma\in E$. Then
\begin{equation}\label{eqn: GtoI}
G_E(s,p^r) = \frac{p}{p-1} \prod_{\gamma \in E } (1-p^{-s-\gamma}) \sum_{j=0}^{\infty} \frac{I_{E,\{1-s\}}(p^{j+r})}{p^{js}}.
\end{equation}
Furthermore, we have
\begin{equation}\label{eqn: GtoIbound}
G_E(s,p^r) \ll_{|E|,\varepsilon} p^{r(c+\varepsilon)}
\end{equation}
for arbitrarily small $\varepsilon$, where $c=-\min(\{1-\re(s)\}\cup \{\re(\gamma):\gamma\in E\})$.
\end{lemma}
\begin{proof}
The assumption $\re(s+\gamma)>0$ and the bound \eqref{divisorbound} imply the absolute convergence of both the right-hand side of \eqref{eqn: GtoI} and the definition \eqref{smallgedef} of $g_E(s,p^r)$ (note that the analogue of \eqref{divisorbound} also holds for $\tau_E$ since $\tau_E=I_{E,\emptyset}$, where $\emptyset$ is the empty set). If $r=0$, then the definition \eqref{bigGEdef} implies $G_E(s,1)=1$, while the right-hand side of \eqref{eqn: GtoI} also equals $1$ because
$$
\sum_{j=0}^{\infty} \frac{I_{E,\{1-s\}}(p^{j+r})}{p^{js}} = \left( 1-\frac{1}{p}\right) \prod_{\gamma\in E}\left( 1- \frac{1}{p^{s+\gamma}}\right)^{-1}
$$
by the definition \eqref{IABdef} and the Euler product expression for zeta. Thus \eqref{eqn: GtoI} holds for $r=0$.

Now suppose that $r\geq 1$. The definition \eqref{bigGEdef} of $G_E(s,p^r)$ implies that
\begin{equation}\label{GEexpand}
G_E(s,p^r) = \left( \frac{p}{p-1}\right) g_E(s,p^r) - \left(\frac{p^s}{p-1}\right) g_E(s,p^{r-1}).
\end{equation}
The definition \eqref{smallgedef} implies that
\begin{equation*}
g_{E}(s,p^j) =  \Bigg\{ \prod_{\gamma\in E} \left( 1- \frac{1}{p^{s+\gamma}}\right)   \Bigg\} \sum_{m=0}^{\infty} \frac{\tau_E (p^{m+j}) }{p^{ms}} 
\end{equation*}
for $j\geq 1$. Note that this also holds for $j=0$ since the definition \eqref{taudef} and the Euler product expression for zeta imply that
$$
\sum_{m=0}^{\infty} \frac{\tau_E (p^{m}) }{p^{ms}} = \prod_{\gamma\in E}\left( 1- \frac{1}{p^{s+\gamma}}\right)^{-1}.
$$
It follows from these and \eqref{GEexpand} that
\begin{align}
G_E(s,p^r)
& = \left( \frac{p}{p-1}\right) \prod_{\gamma\in E} \left( 1-\frac{1}{p^{s+\gamma}}\right) \sum_{m=0}^{\infty}\frac{\tau_E(p^{m+r}) }{p^{ms}} \notag \\
& \hspace{.25in} - \left(\frac{p^s}{p-1}\right)\prod_{\alpha\in A} \left( 1-\frac{1}{p^{s+\gamma}}\right) \sum_{m=0}^{\infty}\frac{\tau_E(p^{m+r-1}) }{p^{ms}} \notag\\
& = \left( \frac{p}{p-1}\right) \prod_{\gamma\in E} \left( 1-\frac{1}{p^{s+\gamma}}\right) \sum_{m=0}^{\infty}\frac{\tau_E(p^{m+r})-p^{s-1}\tau_E(p^{m+r-1}) }{p^{ms}}. \label{GEexpand2}
\end{align}
Now \eqref{IABdef2}, which also holds for $\tau_E$ since $\tau_E=I_{E,\emptyset}$, implies
$$
\tau_E(p^{m+r})-p^{s-1}\tau_E(p^{m+r-1}) = I_{E,\{1-s\}}(p^{m+r}).
$$
The identity \eqref{eqn: GtoI} follows from this and \eqref{GEexpand2}.

The bound \eqref{eqn: GtoIbound} follows from \eqref{eqn: GtoI} and \eqref{divisorbound}.

\end{proof}

We now prove the main result of this section. Recall the notations $E_s:=\{\gamma+s: \gamma\in E\}$ and $E^-:=\{-\gamma:\gamma\in E\}$ that were stated below \eqref{recipe}.

\begin{theorem}\label{euler}
Suppose that $\ell$ is a positive integer, $\xi,\eta$ are complex numbers, and $A,B,Z,W$ are finite multisets of complex numbers with $Z=\{z_1,\dots,z_{\ell}\}$ and $W=\{w_1,\dots,w_{\ell}\}$. Let $A=A_1\cup \cdots \cup A_{\ell}$ and $B=B_1\cup \cdots \cup B_{\ell}$ be partitions of $A$ and $B$, respectively. Let $\epsilon>0$ be arbitrarily small, and suppose that $\re(\xi)=\re(\eta)=2\epsilon$ and $|\re(\gamma)|\leq \epsilon$ for all $\gamma\in A\cup B\cup Z\cup W$. Then
\begin{align*}
& \sum_{\substack{M_1+\cdots +M_{\ell} = N_1+\cdots +N_{\ell} \\ \min\{M_j,N_j\}=0 \ \forall j}} \prod_{j=1}^{\ell} \Bigg\{ p^{N_j(-\frac{1}{2}+\xi-z_j)+M_j(-\frac{1}{2}+\eta-w_j)} \\
& \hspace{.75in}\times \sum_{d=0}^{\infty}  \sum_{q=0}^1 (-1)^q p^{\min\{q+d,N_j\}(1+z_j)+\min\{q+d,M_j\}(1+w_j)-q(2+z_j+w_j)-d(1+\xi+\eta)} \\
& \hspace{1.25in} \times G_{A_j}\Big( 1+z_j, p^{q+d-\min\{q+d,N_j\}}\Big) G_{B_j}\Big( 1+w_j, p^{q+d-\min\{q+d,M_j\}}\Big)  \Bigg\}\\
& = \left( 1-\frac{1}{p}\right)^{-2\ell}\prod_{j=1}^{\ell} \Bigg\{\left(1-\frac{1}{p^{1+w_j+z_j-\xi-\eta}}\right) \prod_{\alpha \in A_j } \left(1-\frac{1}{p^{1+z_j+\alpha}}\right) \prod_{\beta \in B_j } \left(1-\frac{1}{p^{1+w_j+\beta}}\right) \Bigg\}   \\
& \hspace{1in} \times \sum_{m=0}^{\infty} \frac{I_{A_{\xi+\eta}\cup W,(Z^-)_{\xi+\eta}}(p^{m }) I_{B \cup  Z_{-\xi-\eta} , W^{-} }(p^{m})}{ p^{m}} .
\end{align*}
\end{theorem}
\begin{proof}
For brevity, let LHS denote the left-hand side of the conclusion of Theorem~\ref{euler}. We apply Lemma~\ref{GtoI} and deduce that
\begin{align*}
\text{LHS} = &  \sum_{\substack{M_1+\cdots +M_{\ell} = N_1+\cdots +N_{\ell} \\ \min\{M_j,N_j\}=0 \ \forall j}} \prod_{j=1}^{\ell} \Bigg\{ p^{N_j(-\frac{1}{2}+\xi-z_j)+M_j(-\frac{1}{2}+\eta-w_j)} \\
& \hspace{.5in} \times \left( 1-\frac{1}{p}\right)^{-2}\prod_{\alpha \in A_j } \left(1-\frac{1}{p^{1+z_j+\alpha}}\right) \prod_{\beta \in B_j } \left(1-\frac{1}{p^{1+w_j+\beta}}\right)\\
& \times \sum_{d=0}^{\infty} \sum_{q=0}^1 \sum_{i=0}^{\infty} \sum_{k=0}^{\infty} (-1)^q I_{A_j,\{-z_j\}}(p^{i+q+d-\min\{q+d,N_j\}})I_{B_j,\{-w_j\}}(p^{k+q+d-\min\{q+d,M_j\}}) \\
&  \hspace{.5in} \times p^{-i(1+z_j) -k(1+w_j) + \min\{q+d,N_j\}(1+z_j)+\min\{q+d,M_j\}(1+w_j)-q(2+z_j+w_j)-d(1+\xi+\eta)}  \Bigg\}  
\end{align*}
The $d,q,i,k$-sum equals $\Sigma((A_j)_{\xi+\eta},B_j,z_j-\xi-\eta,w_j,M_j,N_j;p)$ by \eqref{SigmaABdef} and Lemma~\ref{IABlemma}(i). Thus Lemma~\ref{lemma2CKV} gives
\begin{align*}
\text{LHS} = &   \sum_{\substack{M_1+\cdots +M_{\ell} = N_1+\cdots +N_{\ell} \\ \min\{M_j,N_j\}=0 \ \forall j}} \prod_{j=1}^{\ell} \Bigg\{ p^{N_j(-\frac{1}{2}-\eta  )+M_j(-\frac{1}{2}+\eta  )} \\
& \hspace{.5in} \times \left( 1-\frac{1}{p}\right)^{-2}\left(1-\frac{1}{p^{1+w_j+z_j-\xi-\eta}}\right)\prod_{\alpha \in A_j } \left(1-\frac{1}{p^{1+z_j+\alpha}}\right) \prod_{\beta \in B_j } \left(1-\frac{1}{p^{1+w_j+\beta}}\right)\\
& \hspace{1in} \times  \sum_{r=0}^{\infty} \frac{I_{(A_j)_{\xi+\eta}\cup \{w_j\},\{-z_j+\xi+\eta\}}(p^{r+M_j}) I_{B_j\cup \{z_j-\xi-\eta\},\{-w_j\}}(p^{r+N_j})}{ p^r}  \Bigg\}  
\end{align*}
We rearrange the factors and use the fact that $M_1+\cdots +M_{\ell} = N_1+\cdots +N_{\ell}$ to arrive at
\begin{align*}
\text{LHS} =  \left( 1-\frac{1}{p}\right)^{-2\ell}
& \prod_{j=1}^{\ell} \Bigg\{\left(1-\frac{1}{p^{1+w_j+z_j-\xi-\eta}}\right)  \prod_{\alpha \in A_j } \left(1-\frac{1}{p^{1+z_j+\alpha}}\right) \prod_{\beta \in B_j } \left(1-\frac{1}{p^{1+w_j+\beta}}\right) \Bigg\}    \\
& \times \sum_{\substack{M_1+\cdots +M_{\ell} = N_1+\cdots +N_{\ell} \\ \min\{M_j,N_j\}=0 \ \forall j}} \prod_{j=1}^{\ell} \Bigg\{ p^{ -\frac{1}{2}N_j    -\frac{1}{2} M_j } \\
& \hspace{.5in} \times  \sum_{r=0}^{\infty} \frac{I_{(A_j)_{\xi+\eta}\cup \{w_j\},\{-z_j+\xi+\eta\}}(p^{r+M_j}) I_{B_j\cup \{z_j-\xi-\eta\},\{-w_j\}}(p^{r+N_j})}{ p^r}  \Bigg\}.
\end{align*}
To evaluate the sum over the $M_j,N_j$, we interchange the order of summation to write it as
\begin{align*}
\sum_{r_1=0}^{\infty}\cdots\sum_{r_{\ell}=0}^{\infty}\sum_{\substack{M_1+\cdots +M_{\ell} = N_1+\cdots +N_{\ell} \\ \min\{M_j,N_j\}=0 \ \forall j}} \prod_{j=1}^{\ell}\Bigg\{    \frac{I_{(A_j)_{\xi+\eta}\cup \{w_j\},\{-z_j+\xi+\eta\}}(p^{r_j+M_j}) I_{B_j\cup \{z_j-\xi-\eta\},\{-w_j\}}(p^{r_j+N_j})}{ p^{r_j+\frac{1}{2}N_j +\frac{1}{2}M_j }}  \Bigg\} .
\end{align*}
We make the change of variables $m_j=M_j+r_j $ and $n_j =N_j+r_j$, and then evaluate the $r_j$-sums to see that this sum equals 
\begin{align*}
& \sum_{r_1=0}^{\infty}\cdots\sum_{r_{\ell}=0}^{\infty}\sum_{\substack{m_1+\cdots +m_{\ell} = n_1+\cdots +n_{\ell} \\ \min\{m_j,n_j\}=r_j \ \forall j}} \prod_{j=1}^{\ell}\Bigg\{     \frac{I_{(A_j)_{\xi+\eta}\cup \{w_j\},\{-z_j+\xi+\eta\}}(p^{m_j}) I_{B_j\cup \{z_j-\xi-\eta\},\{-w_j\}}(p^{n_j})}{ p^{\frac{1}{2}n_j +\frac{1}{2}m_j }}  \Bigg\} \\
& = \sum_{ m_1+\cdots +m_{\ell} = n_1+\cdots +n_{\ell}  } \prod_{j=1}^{\ell}\Bigg\{     \frac{I_{(A_j)_{\xi+\eta}\cup \{w_j\},\{-z_j+\xi+\eta\}}(p^{m_j}) I_{B_j\cup \{z_j-\xi-\eta\},\{-w_j\}}(p^{n_j})}{ p^{\frac{1}{2}n_j +\frac{1}{2}m_j }}  \Bigg\} \\
& = \sum_{m=0}^{\infty} \frac{I_{A_{\xi+\eta}\cup W,(Z^-)_{\xi+\eta}}(p^{m }) I_{B \cup  Z_{-\xi-\eta} , W^{-} }(p^{m})}{ p^{m}} 
\end{align*}
because, by \eqref{IABdef}, $I_{A_{\xi+\eta}\cup W,(Z^-)_{\xi+\eta}}$ is the Dirichlet convolution
$$
I_{A_{\xi+\eta}\cup W,(Z^-)_{\xi+\eta}} = I_{(A_1)_{\xi+\eta}\cup \{w_1\},\{-z_1+\xi+\eta\}} * \cdots * I_{(A_{\ell})_{\xi+\eta}\cup \{w_{\ell}\},\{-z_{\ell}+\xi+\eta\}} ,
$$
and similarly for $I_{B \cup  Z_{-\xi-\eta} , W^{-} }$.
\end{proof}

\section{Vandermonde integral expressions}\label{sec: vandermondeintegral}
We now return to the task of evaluating \eqref{applydeltamethod4}. By multiplicativity, we may formally factor the sum on the right-hand side of \eqref{applydeltamethod4} and write the Euler product formula
\begin{align}
& \sum_{\substack{M_1\cdots M_{\ell} = N_1\cdots N_{\ell} \\ (M_j,N_j)=1 \ \forall j}} \prod_{j=1}^{\ell} \Bigg\{ N_j^{-\frac{1}{2}+\xi-z_j}M_j^{-\frac{1}{2}+\eta-w_j}    \notag\\
& \hspace{.75in} \times \sum_{d=1}^{\infty} \frac{1}{d^{1+\xi+\eta}} \sum_{q=1}^{\infty} \frac{\mu (q) (qd,N_j)^{1+z_j} (qd,M_j)^{1+w_j}}{q^{2+z_j+w_j}} \notag\\
& \hspace{1in} \times G_{A_j}\left( 1+z_j,\frac{qd}{(qd,N_j)} \right) G_{B_j}\left( 1+w_j,\frac{qd}{(qd,M_j)} \right)  \Bigg\} \notag\\
& = \prod_p \Bigg( \sum_{\substack{M_1+\cdots +M_{\ell} = N_1+\cdots +N_{\ell} \\ \min\{M_j,N_j\}=0 \ \forall j}} \prod_{j=1}^{\ell} \Bigg\{ p^{N_j(-\frac{1}{2}+\xi-z_j )+M_j(-\frac{1}{2}+\eta-w_j )} \notag\\
& \hspace{.75in}\times \sum_{d=0}^{\infty}  \sum_{q=0}^1 (-1)^q p^{\min\{q+d,N_j\}(1+z_j)+\min\{q+d,M_j\}(1+w_j)-q(2+z_j+w_j)-d(1+\xi+\eta)} \notag\\
& \hspace{1in} \times G_{A_j}\Big( 1+z_j, p^{q+d-\min\{q+d,N_j\}}\Big) G_{B_j}\Big( 1+w_j, p^{q+d-\min\{q+d,M_j\}}\Big)  \Bigg\} \Bigg) . \label{sumaseulerproduct}
\end{align}
We may use Theorem~\ref{euler} to evaluate each local factor on the right-hand side of \eqref{sumaseulerproduct} and thus write the Euler product as
\begin{align*}
\prod_p \Bigg( \left( 1-\frac{1}{p}\right)^{-2\ell}\prod_{j=1}^{\ell} \Bigg\{\left(1-\frac{1}{p^{1+w_j+z_j-\xi-\eta}}\right) \prod_{\alpha \in A_j } \left(1-\frac{1}{p^{1+z_j+\alpha}}\right) \prod_{\beta \in B_j } \left(1-\frac{1}{p^{1+w_j+\beta}}\right) \Bigg\}   \\
\times \sum_{m=0}^{\infty} \frac{I_{A_{\xi+\eta}\cup W,(Z^-)_{\xi+\eta}}(p^{m }) I_{B \cup  Z_{-\xi-\eta} , W^{-} }(p^{m})}{ p^{m }} \Bigg).
\end{align*}
From this and \eqref{applydeltamethod4}, we arrive at the prediction
\begin{align}
\mathcal{S}_{\ell} \sim & \frac{1}{(\ell!)^2 (2\pi i )^2} \int_{(2\varepsilon)} \int_{(2\varepsilon)} \widetilde{\Upsilon}(\xi) \widetilde{\Upsilon}(\eta) \frac{X^{\xi + \eta}}{T} \int_0^{\infty} \psi\left( \frac{t}{T}\right) \frac{1}{(2\pi i)^{2\ell}} \oint_{|z_1|=\varepsilon} \cdots \oint_{|z_{\ell}|=\varepsilon}   \notag\\
& \times  \oint_{|w_1|=\varepsilon} \cdots \oint_{|w_{\ell}|=\varepsilon} \prod_{j=1}^{\ell} \Bigg\{  \chi (\tfrac{1}{2} +\xi-z_j+it) \chi (\tfrac{1}{2} +\eta-w_j-it)\notag\\
& \hspace{.75in} \times   \zeta(1+z_j+w_j-\xi-\eta) \prod_{\alpha\in A_j } \zeta(1+z_j+\alpha)  \prod_{\beta\in B_j } \zeta(1+w_j+\beta) \Bigg\}\notag\\
& \times  \prod_p \Bigg( \prod_{j=1}^{\ell} \Bigg\{\left(1-\frac{1}{p^{1+w_j+z_j-\xi-\eta}}\right) \prod_{\alpha \in A_j } \left(1-\frac{1}{p^{1+z_j+\alpha}}\right) \prod_{\beta \in B_j } \left(1-\frac{1}{p^{1+w_j+\beta}}\right) \Bigg\}   \notag\\
& \times \left( 1-\frac{1}{p}\right)^{-2\ell}\sum_{m=0}^{\infty} \frac{I_{A_{\xi+\eta}\cup W,(Z^-)_{\xi+\eta}}(p^{m }) I_{B \cup  Z_{-\xi-\eta} , W^{-} }(p^{m})}{ p^{m }} \Bigg) \,dw_{\ell}\cdots dw_1 \,dz_{\ell}\cdots dz_1\,d\xi\,d\eta. \notag
\end{align}
By the Euler product expression for zeta, we may write our prediction more concisely as
\begin{align}
\mathcal{S}_{\ell} \sim & \frac{1}{(\ell!)^2 (2\pi i )^2} \int_{(2\varepsilon)} \int_{(2\varepsilon)} \widetilde{\Upsilon}(\xi) \widetilde{\Upsilon}(\eta) \frac{X^{\xi + \eta}}{T} \int_0^{\infty} \psi\left( \frac{t}{T}\right) \frac{1}{(2\pi i)^{2\ell}} \oint_{|z_1|=\varepsilon} \cdots \oint_{|z_{\ell}|=\varepsilon}   \notag\\
& \hspace{.5in} \times \oint_{|w_1|=\varepsilon} \cdots \oint_{|w_{\ell}|=\varepsilon} \prod_{j=1}^{\ell} \Big\{  \chi (\tfrac{1}{2} +\xi-z_j+it) \chi (\tfrac{1}{2} +\eta-w_j-it)\Big\}\notag\\
& \hspace{.75in} \times \prod_p \Bigg\{\left( 1-\frac{1}{p}\right)^{-2\ell}\sum_{m=0}^{\infty} \frac{I_{A_{\xi+\eta}\cup W,(Z^-)_{\xi+\eta}}(p^{m }) I_{B \cup  Z_{-\xi-\eta} , W^{-} }(p^{m})}{ p^{m }} \Bigg\} \notag \\
& \hspace{1.15in} \times\,dw_{\ell}\cdots dw_1 \,dz_{\ell}\cdots dz_1\,d\xi\,d\eta, \label{applyeulerproductthm}
\end{align}
where the latter Euler product is to be interpreted as its analytic continuation.

To evaluate the $z_j$- and $w_j$-integrals, we need to write out this analytic continuation and determine its poles and residues. The local factor
\begin{equation*}
\left( 1-\frac{1}{p}\right)^{-2\ell}\sum_{m=0}^{\infty} \frac{I_{A_{\xi+\eta}\cup W,(Z^-)_{\xi+\eta}}(p^{m }) I_{B \cup  Z_{-\xi-\eta} , W^{-} }(p^{m})}{ p^{m}}
\end{equation*}
may be written as a power series in $1/p$. The coefficient of $1/p$ in this power series is
\begin{equation*}
2\ell + I_{A_{\xi+\eta}\cup W,(Z^-)_{\xi+\eta}}(p) I_{B \cup  Z_{-\xi-\eta} , W^{-} }(p),
\end{equation*}
which, by \eqref{IABdef2}, equals
\begin{align*}
2\ell +
& \Bigg(\sum_{\alpha\in A} p^{-\alpha-\xi-\eta} +\sum_{j=1}^{\ell} p^{-w_j} - \sum_{j=1}^{\ell} p^{z_j-\xi-\eta} \Bigg) \Bigg(\sum_{\beta\in B} p^{-\beta} + \sum_{j=1}^{\ell} p^{-z_j+\xi+\eta} -\sum_{j=1}^{\ell} p^{w_j}\Bigg) \\
& =  \sum_{\alpha\in A}\sum_{\beta\in B}p^{-\alpha-\beta-\xi-\eta} + \sum_{\alpha\in A} \sum_{j=1}^{\ell}   p^{-\alpha-z_j}  - \sum_{\alpha\in A}\sum_{j=1}^{\ell} p^{-\alpha-\xi-\eta+w_j}  \\
& \hspace{.25in} + \sum_{\beta\in B}\sum_{j=1}^{\ell} p^{-\beta-w_j} + \sum_{i=1}^{\ell}\sum_{j=1}^{\ell} p^{-z_i-w_j+\xi+\eta}  -\sum_{1\leq i\neq j\leq \ell} p^{w_i-w_j} \\
& \hspace{.25in}  - \sum_{\beta\in B}\sum_{j=1}^{\ell} p^{-\beta-\xi-\eta+z_j} -\sum_{1\leq i\neq j\leq {\ell}} p^{z_i-z_j} + \sum_{i=1}^{\ell}\sum_{j=1}^{\ell} p^{z_i+w_j-\xi-\eta}.
\end{align*}
Therefore the (analytic continuation of the) Euler product in \eqref{applyeulerproductthm} may be written as
\begin{align}
& \prod_{\substack{\alpha\in A \\ \beta\in B}} \zeta(1+\alpha+\beta+\xi+\eta)\prod_{\substack{1\leq j\leq \ell \\ \alpha\in A}}\zeta(1+\alpha+z_j) \prod_{\substack{1\leq j\leq \ell \\ \beta\in B}}\zeta(1+\beta+w_j) \notag\\
& \hspace{.25in} \times \prod_{\substack{1\leq j\leq \ell \\ \alpha\in A}} (1/\zeta)(1+ \alpha+\xi+\eta-w_j)\prod_{\substack{1\leq j\leq \ell \\ \beta\in B}} (1/\zeta) (1 +\beta+\xi+\eta-z_j)  \notag\\
& \hspace{.5in}  \times \prod_{\substack{1\leq i,j\leq \ell \\ i\neq j}} (1/\zeta) (1-z_i+z_j) \prod_{\substack{1\leq i,j\leq \ell \\ i\neq j}} (1/\zeta) (1-w_i+w_j)  \notag\\
& \hspace{.75in}  \times \prod_{\substack{1\leq i,j\leq \ell}}\zeta(1 +z_i+w_j-\xi-\eta)\zeta(1 -z_i-w_j+\xi+\eta) \notag\\
& \hspace{1in} \times \mathcal{A}(A,B,Z,W,\xi+\eta), \label{vandermondeintegrand}
\end{align}
where $\mathcal{A}(A,B,Z,W,\xi+\eta)$ is an Euler product that converges absolutely whenever $\re(\xi)=\re(\eta)=2\varepsilon$ and $|\re(\gamma)|\leq \varepsilon$ for all $\gamma\in A\cup B\cup Z\cup W$. Explicitly, $\mathcal{A}$ is defined by
\begin{align}
& \mathcal{A}(A,B,Z,W,\xi+\eta) \notag \\
& := \prod_p \Bigg\{\left( 1-\frac{1}{p}\right)^{-2\ell} \prod_{\substack{\alpha\in A \\ \beta\in B}} \left( 1- \frac{1}{p^{1+\alpha+\beta+\xi+\eta}} \right) \prod_{\substack{1\leq j\leq \ell \\ \alpha\in A}}\left( 1- \frac{1}{p^{1+\alpha+z_j}} \right)  \notag\\
& \hspace{.25in} \times \prod_{\substack{1\leq j\leq \ell \\ \beta\in B}} \left( 1- \frac{1}{p^{1+\beta+w_j}} \right) \prod_{\substack{1\leq j\leq \ell \\ \alpha\in A}} \left( 1- \frac{1}{p^{1+\alpha+\xi+\eta-w_j}} \right)^{-1} \notag \\
&  \hspace{.45in} \times \prod_{\substack{1\leq j\leq \ell \\ \beta\in B}} \left( 1- \frac{1}{p^{1+\beta+\xi+\eta-z_j}} \right)^{-1}\prod_{\substack{1\leq i,j\leq \ell \\ i\neq j}} \left( 1- \frac{1}{p^{1-z_i+z_j}} \right)^{-1}  \notag \\
& \hspace{.65in} \times  \prod_{\substack{1\leq i,j\leq \ell \\ i\neq j}} \left( 1- \frac{1}{p^{1-w_i+w_j}} \right)^{-1} \prod_{\substack{1\leq i,j\leq \ell}} \left( 1- \frac{1}{p^{1+z_i+w_j-\xi-\eta}} \right)  \notag\\
& \hspace{.85in} \times  \prod_{\substack{1\leq i,j\leq \ell}} \left( 1- \frac{1}{p^{1-z_i-w_j+\xi+\eta}} \right)\sum_{m=0}^{\infty} \frac{I_{A_{\xi+\eta}\cup W,(Z^-)_{\xi+\eta}}(p^{m }) I_{B \cup  Z_{-\xi-\eta} , W^{-} }(p^{m})}{ p^{m }} \Bigg\}, \label{mathcalAdef}
\end{align}
where we recall that $I_{C,D}$ is defined by \eqref{IABdef}. This leads us to conjecture Conjecture~\ref{conj: vandermonde}.

We now evaluate the $z_j$- and $w_j$- integrals in Conjecture~\ref{conj: vandermonde} (or equivalently in \eqref{applyeulerproductthm}). For convenience, we assume that the elements of $A$ are distinct from each other, and the elements of $B$ are distinct from each other. If $\re(\xi)=\re(\eta)=2\varepsilon$ and $|\alpha|,|\beta|\leq \varepsilon/2$ for all $\alpha\in A$ and $\beta\in B$, then the poles of the integrand displayed in Conjecture~\ref{conj: vandermonde}, viewed as a function of $z_j$ (resp.~$w_j$), that are enclosed by the circle $|z_j|=\varepsilon$ (resp.~$|w_j|=\varepsilon$) are at the points $z_j=-\alpha$, where $\alpha \in A$ (resp.~$w_j=-\beta$, where $\beta\in B$). Thus the value of the $z_1,\dots,z_{\ell},w_1,\dots,w_{\ell}$-integral in Conjecture~\ref{conj: vandermonde} equals the sum of the residues of the integrand at the points
\begin{equation}\label{poles}
(z_1,\dots,z_{\ell},w_1,\dots,w_{\ell}) = (-\alpha_1,\dots,-\alpha_{\ell},-\beta_1,\dots,-\beta_{\ell}),
\end{equation}
where $\alpha_1,\dots,\alpha_{\ell}\in A$ and $\beta_1,\dots,\beta_{\ell}\in B$. If $\alpha_i=\alpha_j$ for some $i\neq j$, then the residue is zero because of the presence of the factor $1/\zeta(1-z_i+z_j)$ in Conjecture~\ref{conj: vandermonde}. Hence the residue is nonzero only at the points \eqref{poles} such that $\{\alpha_1,\dots,\alpha_{\ell}\}$ is an $\ell$-element subset $U$, say, of $A$, and $\{\beta_1,\dots,\beta_{\ell}\}$ is an $\ell$-element subset $V$, say, of $B$. At such a point, the residue of \eqref{vandermondeintegrand} equals
\begin{align}
& \prod_{\substack{\alpha\in A \\ \beta\in B}} \zeta(1+\alpha+\beta+\xi+\eta)\prod_{\substack{ \alpha\in A,\hat{\alpha}\in U \\ \alpha\neq \hat{\alpha}}}\zeta(1+\alpha-\hat{\alpha}) \prod_{\substack{ \beta\in B,\hat{\beta}\in V \\ \beta\neq \hat{\beta}}}\zeta(1+\beta-\hat{\beta}) \notag\\
& \hspace{.25in} \times \prod_{\substack{\hat{\beta}\in V \\ \alpha\in A}} (1/\zeta)(1+ \alpha+\xi+\eta+\hat{\beta})\prod_{\substack{\hat{\alpha}\in U \\ \beta\in B}} (1/\zeta) (1 +\beta+\xi+\eta+\hat{\alpha})  \notag\\
& \hspace{.5in}  \times \prod_{\substack{\alpha,\hat{\alpha}\in U \\ \alpha\neq\hat{\alpha}}} (1/\zeta) (1+\alpha-\hat{\alpha}) \prod_{\substack{\beta,\hat{\beta}\in V \\ \beta\neq\hat{\beta}}} (1/\zeta) (1+\beta-\hat{\beta})  \notag\\
& \hspace{.75in}  \times \prod_{\substack{\hat{\alpha}\in U\\ \hat{\beta}\in V}}\zeta(1 -\hat{\alpha}-\hat{\beta}-\xi-\eta)\zeta(1 +\hat{\alpha}+\hat{\beta}+\xi+\eta) \notag\\
& \hspace{1in} \times \mathcal{A}(A,B,U,V,\xi+\eta), \label{residue}
\end{align}
with
\begin{align}
& \mathcal{A}(A,B,U,V,\xi+\eta) \notag \\
& = \prod_p \Bigg\{ \prod_{\substack{\alpha\in A \\ \beta\in B}} \left( 1- \frac{1}{p^{1+\alpha+\beta+\xi+\eta}} \right) \prod_{\substack{ \alpha\in A,\hat{\alpha}\in U \\ \alpha\neq \hat{\alpha}}} \left( 1- \frac{1}{p^{1+\alpha- \hat{\alpha} }} \right)  \notag\\
& \hspace{.25in} \times \prod_{\substack{ \beta\in B,\hat{\beta}\in V \\ \beta\neq \hat{\beta}}} \left( 1- \frac{1}{p^{1+\beta-\hat{\beta}}} \right) \prod_{\substack{\hat{\beta}\in V \\ \alpha\in A}} \left( 1- \frac{1}{p^{1+\alpha+\xi+\eta+ \hat{\beta} }} \right)^{-1} \notag \\
&  \hspace{.45in} \times \prod_{\substack{\hat{\alpha}\in U \\ \beta\in B}} \left( 1- \frac{1}{p^{1+\beta+\xi+\eta+\hat{\alpha}}} \right)^{-1}\prod_{\substack{\alpha,\hat{\alpha}\in U \\ \alpha\neq\hat{\alpha}}} \left( 1- \frac{1}{p^{1+\alpha -\hat{\alpha} }} \right)^{-1}  \notag \\
& \hspace{.65in} \times  \prod_{\substack{\beta,\hat{\beta}\in V \\ \beta\neq\hat{\beta}}} \left( 1- \frac{1}{p^{1+\beta-\hat{\beta}}} \right)^{-1} \prod_{\substack{\hat{\alpha}\in U\\ \hat{\beta}\in V}} \left( 1- \frac{1}{p^{1-\hat{\alpha} -\hat{\beta}-\xi-\eta}} \right)  \notag\\
& \hspace{.85in} \times  \prod_{\substack{\hat{\alpha}\in U\\ \hat{\beta}\in V}} \left( 1- \frac{1}{p^{1+\hat{\alpha}+\hat{\beta}+\xi+\eta}} \right)\sum_{m=0}^{\infty} \frac{ \tau_{(A \smallsetminus U)_{\xi+\eta}\cup V^- }(p^m) \tau_{ B\smallsetminus V \cup  (U_{\xi+\eta})^-}(p^m) }{ p^{m }} \Bigg\} \label{vandermondeintegrand2}
\end{align}
because
\begin{align*}
\left( 1-\frac{1}{p}\right)^{-2\ell} \prod_{\substack{ \alpha\in A,\hat{\alpha}\in U  }} \left( 1- \frac{1}{p^{1+\alpha- \hat{\alpha} }} \right) \prod_{\substack{ \beta\in B,\hat{\beta}\in V  }} \left( 1- \frac{1}{p^{1+\beta-\hat{\beta}}} \right) \\
=\prod_{\substack{ \alpha\in A,\hat{\alpha}\in U \\ \alpha\neq \hat{\alpha}}} \left( 1- \frac{1}{p^{1+\alpha- \hat{\alpha} }} \right) \prod_{\substack{ \beta\in B,\hat{\beta}\in V \\ \beta\neq \hat{\beta}}} \left( 1- \frac{1}{p^{1+\beta-\hat{\beta}}} \right)
\end{align*}
and
\begin{equation*}
I_{A_{\xi+\eta}\cup V^-,U_{\xi+\eta}}(p^{m }) I_{B \cup  (U^-)_{-\xi-\eta} , V }(p^{m}) = \tau_{(A \smallsetminus U)_{\xi+\eta}\cup V^- }(p^m) \tau_{ B\smallsetminus V \cup  (U_{\xi+\eta})^-}(p^m)
\end{equation*}
by \eqref{IABdef}. We may write this residue more concisely as
\begin{equation*}
\prod_p \Bigg\{ \sum_{m=0}^{\infty} \frac{ \tau_{(A \smallsetminus U)_{\xi+\eta}\cup V^- }(p^m) \tau_{ B\smallsetminus V \cup  (U_{\xi+\eta})^-}(p^m) }{ p^{m }} \Bigg\},
\end{equation*}
or, similarly,
\begin{equation*}
\sum_{n=1}^{\infty} \frac{ \tau_{(A \smallsetminus U)_{\xi+\eta}\cup V^- }(n) \tau_{ B\smallsetminus V \cup  (U_{\xi+\eta})^-}(n) }{ n },
\end{equation*}
which we interpret as its analytic continuation, as the reciprocals of the Euler product expressions for the zeta functions in \eqref{residue} are precisely those found in \eqref{vandermondeintegrand2}. Now for each pair $U,V$ of sets such that $U\subseteq A$ and $V\subseteq B$ with $|U|=|V|=\ell$, the number of points \eqref{poles} with $U=\{\alpha_1,\dots,\alpha_{\ell}\}$ and $V=\{\beta_1,\dots,\beta_{\ell}\}$ is $(\ell !)^2$. Thus, evaluating the $z_j$- and $w_j$-integrals in Conjecture~\ref{conj: vandermonde} leads to the prediction
\begin{align}
\mathcal{S}_{\ell} \sim
& \sum_{\substack{ U\subseteq A, V\subseteq B\\ |U|=|V|={\ell} }}\frac{1}{(2\pi i )^2} \int_{(2\varepsilon)} \int_{(2\varepsilon)} \widetilde{\Upsilon}(\xi) \widetilde{\Upsilon}(\eta) \frac{X^{\xi + \eta}}{T} \int_0^{\infty} \psi\left( \frac{t}{T}\right)  \notag\\
& \hspace{.25in} \times \prod_{\alpha\in U}   \chi (\tfrac{1}{2} +\xi+\alpha+it) \prod_{\beta\in V}\chi (\tfrac{1}{2} +\eta+\beta-it)  \notag\\
& \hspace{.5in}  \times \sum_{n=1}^{\infty} \frac{ \tau_{(A \smallsetminus U)_{\xi+\eta}\cup V^- }(n) \tau_{ B\smallsetminus V \cup  (U_{\xi+\eta})^-}(n) }{ n }\,d\xi\,d\eta. \notag
\end{align}
This is the same as Conjecture~\ref{conj: ell-swaps} because Lemma~\ref{IABlemma}(i) implies
\begin{equation*}
\tau_{(A \smallsetminus U)_{\xi+\eta}\cup V^- }(n) \tau_{ B\smallsetminus V \cup  (U_{\xi+\eta})^-}(n) =  \tau_{(A \smallsetminus U)_{\xi}\cup (V_{\eta})^- }(n)\tau_{ (B\smallsetminus V)_{\eta} \cup  (U_{\xi})^-}(n).
\end{equation*}

\end{document}